\documentclass[12pt]{amsart}

\usepackage{tikz}
\usepackage{amsmath, verbatim}
\usepackage{amssymb,amsfonts,mathrsfs}
\usepackage[colorlinks=true,linkcolor=blue,citecolor=blue]{hyperref}
\usepackage[driver=pdftex,margin=3cm,heightrounded=true,centering]{geometry}

\usepackage{mllocal}

\numberwithin{equation}{subsection}

\begin{document}

\title[Mapping Properties of the Heat Operator on Edge Manifolds]
{Mapping Properties of the Heat Operator \\ on Edge Manifolds}


\author{Eric Bahuaud}
\address{Department of Mathematics,
Seattle University,
901 12th Avenue,
Seattle, WA 98122,
USA}

\email{bahuaude (at) seattleu (dot) edu}

\author{Emily B.\ Dryden}
\address{Department of Mathematics,
Bucknell University,
Lewisburg, PA 17837,
USA}

\email{emily (dot) dryden (at) bucknell (dot) edu}
\urladdr{http://www.facstaff.bucknell.edu/ed012}

\author{Boris Vertman}
\address{Mathematisches Institut,
Universit\"at Bonn,
53115 Bonn,
Germany}
\email{vertman (at) math (dot) uni-bonn (dot) de}
\urladdr{www.math.uni-bonn.de/people/vertman}

\date{\today}

\begin{abstract}
{We consider the heat operator acting on differential forms on spaces with complete and incomplete edge metrics. 
In the latter case we study the heat operator of the Hodge Laplacian with algebraic boundary conditions at the 
edge singularity. We establish the mapping properties of the heat operator, recovering and extending the classical results from smooth 
manifolds and conical spaces. The estimates, together with strong continuity of the heat operator, yield short-time 
existence of solutions to certain semilinear parabolic equations. Our discussion reviews and generalizes earlier 
work by Jeffres and Loya.}
\end{abstract}

\maketitle

\tableofcontents

\section{Introduction and statement of the main results}

Edge metrics of complete and incomplete type provide interesting examples of geometries that include asymptotically hyperbolic, asymptotically cylindrical and conic spaces.  Incomplete and complete edge metrics behave differently from a geometric perspective.  However, these metrics lend themselves to a parallel approach when one wants to do constructions of an analytic flavour.  For example, to analyze asymptotics of solutions to the heat equation near spatial infinity in the complete case, one must introduce a geometric compactification.  The resulting complete edge metrics, like their incomplete cousins, give rise to differential operators that are degenerate or singular in a sufficiently controlled manner that  generalization of classical results is possible.  In particular, we consider the Hodge Laplacian on manifolds with either type of edge metric and establish regularity properties of solutions to the heat equation.  We study the cases in parallel, reviewing and extending previous work by Jeffres and Loya for conic and b-metrics \cite{JefLoy:RSH, JefLoy:RHM}.  Recall that the Hodge Laplacian on compact manifolds has nonnegative spectrum.

We begin by defining our metrics precisely.  Let $\overline{M}$ be an $m$-dimensional compact manifold with boundary $\partial M$, where $\partial M$ is the total space of a fibration $\phi: \partial M \to B$, and the fibre $F$ and base $B$ are closed manifolds. We consider a defining function $x:C^{\infty}(\overline{M})\to \R^+\cup \{0\}$ of the boundary $\partial M$ with $x^{-1}(0)=\partial M$ and $dx\neq 0$ on $\partial M$. Using the integral curves of $\textup{grad} (x)$ we identify a collar neighbourhood $U\subset \overline{M}$ of the boundary $\partial M$ with $[0,1)\times \partial M$, where $\partial M$ is identified with $\{0\}\times \partial M$. \medskip

\begin{defn}\label{d-edge}
A Riemannian manifold $(\overline{M} \backslash \partial M,g):=(M,g)$ has a simple
\begin{enumerate}
\item incomplete edge at $B$ if, with respect to the identification $U\backslash \partial M \cong (0,1)\times \partial M$, the Riemannian metric $g=g_0+h$ and $g_0$ attains the form 
$$g_0\restriction U\backslash \partial M=dx^2+x^2 g^F+\phi^*g^B,$$
\item complete edge at $B$ if, with respect to the identification $U\backslash \partial M \cong (0,1)\times \partial M$, the Riemannian metric $g=g_0+h$ and $g_0$ attains the form 
$$g_0\restriction U\backslash \partial M=\frac{dx^2}{x^2}+\frac{\phi^*g^B}{x^2} + g^F,$$
\end{enumerate}
where $g^B$ is a Riemannian metric on the closed manifold $B$, $g^F$ is a symmetric 2-tensor on the fibration $\partial M$ restricting to a Riemannian metric on each fibre $F$, and $|h|_{g_0}=O(x)$ as $x \to  0$.  
\end{defn}

We will refer to either of these metrics as an edge metric.  
In the case of an incomplete edge metric, if $\dim F = 0$ then $(M,g)$ is the interior of a compact 
(non-singular Riemannian) manifold with boundary $B\times F$. Hence we 
consider $\dim F \geq 1$ in the incomplete setup.

As the names imply, incomplete or complete edge metrics are incomplete or complete as Riemannian metrics.  In the complete case $\partial M$ is at infinite distance from any point in the interior of $\overline{M}$; note that the two types of edge metrics are conformal in the interior of $\overline{M}$.  Examples of complete edge metrics include asymptotically hyperbolic (conformally compact or $0$-metrics) and asymptotically cylindrical (b-metrics) metrics, with $\dim F = 0$ and $\dim B = 0$ respectively.  The product metric on $\bH^{b+1} \times \bS^f$ provides an example of a complete edge metric with trivial fibration structure.  Examples of incomplete edge metrics include conic metrics and conformal compactifications of asymptotically hyperbolic metrics.

We consider a slightly restricted class of edge metrics, requiring that 
$\phi: (\partial M, g^F + \phi^*g^B) \to (B, g^B)$ be a Riemannian submersion. Recall that
if $p\in \partial M$, then $T_p\partial M$ splits into vertical and horizontal subspaces as $T^V_p \partial M \oplus
T^H_p \partial M$, where $T^V_p\partial M$ is the tangent space to the fibre of $\phi$ through $p$ and $T^H_p \partial M$ is the
orthogonal complement of this subspace. The new condition on $g_0$ 
implies that the restriction of the tensor $g^F$ to $T^H_p \partial M$ vanishes. Moreover, in the case of incomplete edge metrics we need to assume that the Laplacians associated to $g^F$ at each $b\in B$ are isospectral. 
We summarize these additional conditions in the definition below.

\begin{defn}\label{def-feasible}
Let $(M,g)$ be a Riemannian manifold with an edge metric. This metric $g=g_0+h$ is said to be feasible if 
\begin{enumerate}
\item $\phi: (\partial M, g^F + \phi^*g^B) \to (B, g^B)$ is a Riemannian submersion;
\item if the edge metric is incomplete, the Laplacians associated to $g^F$ at each $b\in B$ are isospectral.
\end{enumerate}
\end{defn} 

\begin{remark}
The isospectrality condition ensures polyhomogeneity of the associated heat kernels 
when lifted to the corresponding parabolic blowup space. Our arguments require only conormality of the heat kernels, and fundamentally depend only on the leading orders of the kernels at the various 
boundary faces of the blowup; this is much less than the full polyhomogeneity. 
Nevertheless we need to pose the isospectrality condition in the incomplete edge case within the present discussion, as the actual constructions of the various heat kernels
in \cite{Moo:HMC} and \cite{MazVer:ATM} seem to require polyhomogeneity in an essential way.
\end{remark}

Our perspective for establishing regularity for solutions to the heat equation is as follows: the mapping properties of the corresponding heat operator are encoded in the asymptotic behaviour of the heat kernel on an appropriate blowup of the heat space, as established by Mazzeo and the third author \cite{MazVer:ATM} for incomplete and by Albin \cite{Alb:RIT} for complete edge metrics.  The regularity is discussed in terms of spaces with bounded edge derivatives and appropriate H\"older spaces, which take into account the underlying singular or asymptotic geometry. To make this precise we introduce the notion of edge vector fields $\mathcal{V}_e$.

\begin{defn} Let $\overline{M}$ be a compact manifold with boundary $\partial M$ being the total space of a fibration $\phi:\partial M \to B$ with fibre $F$. Define the vector space $\mathcal{V}_e$ to be the space of vector fields smooth in the interior of $\overline{M}$ and tangent at the boundary $\partial M$ to the fibres of the fibration. 
\end{defn}

The space $\mathcal{V}_e$ is closed under the ordinary Lie bracket of vector fields, hence defines a Lie algebra. In local coordinates $\mathcal{V}_e$ can be described as follows. Let $y=(y_1,...,y_{b}),b=\dim B$ be the local coordinates on $B$ lifted to $\partial M$ and then extended inwards. Let $z=(z_1,...,z_f),f=\dim F$ restrict to local coordinates on $F$ along each fibre of $\partial M$. Then $(x,y,z)$ are the local coordinates on $\overline{M}$ near the boundary and the edge vector fields $\mathcal{V}_e$ are locally generated by 
\[
\left\{x\frac{\partial}{\partial x}, x\frac{\partial}{\partial y_1}, \dots, x \frac{\partial}{\partial y_b}, \frac{\partial}{\partial z_1},\dots, \frac{\partial}{\partial z_f}\right\}.
\]
The tangent bundle ${}^eTM$ is naturally associated to  $\mathcal{V}_e=C^\infty(M,{}^eTM)$, defined by requiring
that the edge vector fields $\mathcal{V}_e$ form a spanning set of sections. The dual to ${}^eTM$ is denoted by ${}^eT^*M$ and is spanned locally by the following set of one-forms:
\begin{align}\label{triv}
\left\{\theta^e_1,\dots, \theta^e_m\right\} = \left\{\frac{dx}{x}, \frac{dy_1}{x}, \dots, \frac{dy_b}{x}, dz_1,\dots,dz_f\right\};
\end{align}
though singular in the usual sense, these one-forms are smooth as sections of ${}^eT^*M$. 
We denote ${}^e\Lambda^p M= \Lambda^p({}^eT^*M)$ and write $C(M,{}^e\Lambda^p M)$ for the space of continuous sections in ${}^e\Lambda^p M$. Similarly, we define ${}^{ie}T^*M$ to be spanned locally by 
\[
\left\{\theta^{ie}_1,\dots, \theta^{ie}_m\right\} = \left\{dx, dy_1, \dots, dy_b, x dz_1,\dots, x dz_f\right\},
\]
and write $C(M,{}^{ie}\Lambda^p M)$ for the space of continuous sections in 
${}^{ie}\Lambda^p M =\Lambda^p({}^{ie}T^*M)$.
In fact 
$$
C(M,{}^{ie}\Lambda^p M) = x^p C(M,{}^{e}\Lambda^p M).
$$
 
\begin{defn}\label{defn:norms}
Let $(M,g)$ be a Riemannian manifold with an incomplete  edge metric. 
\begin{enumerate}
\item 
Let $\w \in C(M,{}^{ie}\Lambda^p M)$ be
given locally by 
$
\w = \sum_{|I|=p} \w_I \, \theta^{ie}_I, \ \w_I \in C(\overline{M}),
$
with each $I=(i_1,\dots,i_p)$ being a multiindex and $\theta^{ie}_I= 
\theta^{ie}_{i_1} \wedge \dots \wedge \theta^{ie}_{i_p}$.
For $k\in \N$, let $C^k_e(M,{}^{ie}\Lambda^p M)$ denote the space of
all $\w \in C(M,{}^{ie}\Lambda^p M)$ such that 
for any choice of edge vector fields $V_j\in \mathcal{V}_e, j \leq k$, each
$V_1\dots V_j\, \w_I \in C(\overline{M})$. 
Locally, we may write
$$V_1 \dots V_j \, \w := \sum_{|I|=p} (V_1 \dots V_j\w_I) \, \theta^{ie}_I.$$
Considering suprema over each local coordinate 
neighbourhood, we put
$\|\w\|= \sup_{|I|=p} \|\w_I\|_\infty,$
and define a norm on $C^k_e(M,{}^{ie}\Lambda^p M)$ by
$$\| \w \|_k := \| \w \| + \sum_{j \leq k}  \| V_1 \dots V_j \w \|.$$
It can be shown that coordinate changes lead to equivalent norms.

\item The H\"older space $C^{\A}_e(M,{}^{ie}\Lambda^p M), \A\in (0,1),$ 
 consists of $\w \in C(M,{}^{ie}\Lambda^p M)$ such that the H\"older norm
$$\| \w \|_\A = \| \w \| + 
\sup_{|I|=p}  \left\|\frac{\w_I(q) -  \w_I(q')}{d(q,q')^\A}\right\|_{\infty, (q,q')\in \overline{M}^2} < \infty,$$
 where $d(q,q')$ represents the distance between $q,q'\in \overline{M}$ 
with respect to the Riemannian metric $g$.  Note that in the local neighbourhood of the edge these distances are uniformly equivalent to
$$d((x,y,z), (\wx, \wy, \wz))^2 \approx |x-\wx|^2+|y-\wy|^2+(x+\wx)^2|z-\wz|^2.$$
\end{enumerate}
\end{defn}

For a Riemannian manifold $(M,g)$ 
with a complete  edge metric, we replace ${}^{ie}\Lambda^p M$ by ${}^{e}\Lambda^p M$ in Definition \ref{defn:norms}; the distance $d(q,q')$ with respect to a complete edge metric $g$ is uniformly
equivalent to 
$$d((x,y,z), (\wx, \wy, \wz))^2 \approx \frac{|x-\wx|^2+|y-\wy|^2}{(x+\wx)^2}+|z-\wz|^2.$$

For an incomplete edge space $(M,g)$ we may also consider the Banach space of continuous sections 
$\mathscr{C}_e^0(M,{}^{ie}\Lambda^p M)$ which are fibrewise constant at $x=0$. 
This is precisely the space of continuous sections with respect to the topology on $M$ induced by the Riemannian metric $g$. 
The corresponding space of continuous $k$-times edge-differentiable sections shall be denoted by $\mathscr{C}^k_e(M,{}^{ie}\Lambda^p M)$, with
\[
 \mathscr{C}^k_e(M,{}^{ie}\Lambda^p M):=\{u \in  C^k_e(M,{}^{ie}\Lambda^p M) 
\mid V_1\cdots V_j u \in \mathscr{C}^0_e(M,{}^{ie}\Lambda^p M) \ \textup{for any} \ V_i \in \mathcal{V}_e \},
\]
which is a Banach subspace of $C^k_e(M,{}^{ie}\Lambda^p M)$. For the H\"older space with fractional differentiability 
we write $\mathscr{C}^{\A}_e(M,{}^{ie}\Lambda^p M):=C^{\A}_e(M,{}^{ie}\Lambda^p M)$.  
Note that similar considerations are possible in the setup of complete edge metrics, but in that case these do not lead to a refinement of the 
regularity statements below.

The heat operator acts between weighted edge spaces.  A differential form $u$ is in a weighted edge space, denoted $u \in x^\beta C^k_e$, 
if and only if $u = x^\beta \w$, with $\w \in C^k_e$.  Our estimates with respect to these spaces are as follows, where a precise definition of the (non-logarithmic) algebraic boundary conditions on incomplete edges is given in \S \ref{section-algebraic}. 

\begin{thm} \label{thm:main:incomplete}
Let $(M,g)$ be an $m$-dimensional Riemannian manifold with a feasible incomplete edge metric. Let $e^{-t\Delta_\Gamma}$ denote the heat operator corresponding to a non-logarithmic algebraic self-adjoint extension $\Delta_\Gamma$ of the Hodge Laplacian on $p$-forms on $(M,g)$. 
Consider $\nu \in \R^+$ such that $(\pm \nu +1/2)$ is an indicial root of the (rescaled) $p$-form Hodge Laplacian.
We write $\nu_{\min}$ for the minimum of these numbers. We denote by $\nu_q$ the maximal number $\nu \in [0,1)$ such that the asymptotic expansion of solutions in the domain of $\Delta_\Gamma$ at the edge admits terms of the form $x^{- \nu +1/2}$ or $\sqrt{x}\log (x)$ for $\nu =0$.
We define a range of weights 
\begin{align*}
&\beta \in (-(\dim F+3)/2 - \nu_{\min}, (1-\dim F)/2 + \nu_{\min}], \\  &\qquad \qquad 
 \qquad \qquad\textup{if $\Delta_\Gamma$ is the Friedrichs extension,} \\
&\beta \in (-(\dim F+3)/2 + \nu_q, (1-\dim F)/2 - \nu_q],  \\ &\qquad  \qquad  \textup{if $\Delta_\Gamma$ is not the Friedrichs extension.} 
\end{align*}
Then $e^{-t\Delta_\Gamma}$ is bounded for any fixed $t \in \R^+$ with estimates of the form 
\begin{equation} 
\label{est:Ck-spaces-inc} 
\begin{split}
&e^{-t\Delta_\Gamma}: x^{\beta}C^k_e(M,{}^{ie}\Lambda^p M) \to x^{\beta}C^{k+n}_e(M,{}^{ie}\Lambda^p M), \\
&\|e^{-t\Delta_\Gamma}\omega \|_{x^{\beta}C^{k+n}_e}\leq C t^{-n/2}\|\omega\|_{x^{\beta}C^k_e}, 
\end{split}
\end{equation}
and with respect to H\"older spaces
\begin{equation}
\label{est:holder-spaces-inc} 
\begin{split}
&e^{-t\Delta_\Gamma}: x^{\beta}C^{\A}_e(M,{}^{ie}\Lambda^p M) \to x^{\beta}C^{2}_e(M,{}^{ie}\Lambda^p M), \\ 
&\|e^{-t\Delta_\Gamma}\omega\|_{x^{\beta }C^{2}_e}\leq C t^{-1+\A/2}\|\omega\|_{x^{\beta}C^{\A}_e}.
\end{split}
\end{equation} 
Precisely the same estimates are true with $C^{k}_e(M,{}^{ie}\Lambda^p M)$ spaces 
replaced by $\mathscr{C}^{k}_e(M,{}^{ie}\Lambda^p M), k\in \N$, if the weight $\beta$
does not attain the right boundary value of the weight interval.
\end{thm}

In the special case of functions (i.e., $p=0$), we have $\nu_{\min}=(\dim F-1)/2$ and the possible weights in Theorem \ref{thm:main:incomplete} are
$\beta \leq 0$. The weight $\beta=0$ is of particular interest in \S \ref{s-short}, and possible only in the action of the Friedrichs heat kernel.
In the mapping properties of the Friedrichs heat kernel on $x^\beta \mathscr{C}^{k}_e(M,{}^{ie}\Lambda^p M)$, however, $\beta=0$ is a priori excluded.
Nevertheless, in the function case harmonic forms on the fibre $F$ are constant,
which leads to additional features of the heat kernel and allows for a refinement of the regularity statement as follows.

\begin{thm} \label{thm:main:incomplete-functions}
Let $(M,g)$ be an $m$-dimensional Riemannian manifold with a feasible incomplete edge metric. 
Let $e^{-t\Delta_{\mathscr{F}}}$ denote the heat operator corresponding to the Friedrichs self-adjoint extension 
$\Delta_{\mathscr{F}}$ of the Laplace-Beltrami operator on $(M,g)$. Then $e^{-t\Delta_{\mathscr{F}}}$ is bounded for any fixed $t \in \R^+$ with estimates of the form
\begin{equation} 
\label{est:Ck-spaces-inc-functions} 
\begin{split}
&e^{-t\Delta_{\mathscr{F}}}: \mathscr{C}^k_e(M) \to \mathscr{C}^{k+n}_e(M), \\
&\|e^{-t\Delta_{\mathscr{F}}}f\|_{C^{k+n}_e}\leq C t^{-n/2}\|f\|_{C^k_e}. 
\end{split}
\end{equation}
\end{thm}

In the case of complete edge metrics, we have the following estimates. 

\begin{thm} \label{thm:main:complete}
Let $(M,g)$ be an $m$-dimensional Riemannian manifold with a feasible complete edge metric. Let $e^{-t\Delta}$ denote the heat operator of the unique self-adjoint extension $\Delta$ of the associated Hodge Laplacian on $p$-forms on $(M,g)$. Then for any $w \in \bR$, $e^{-t\Delta}$ is bounded for any fixed $t\in \R^+$ with estimates of the form
\begin{equation} 
\label{est:Ck-spaces} 
\begin{split}
&e^{-t\Delta}: x^{\beta} C^k_e(M,{}^{e}\Lambda^p M) \to x^{\beta} C^{k+n}_e(M,{}^{e}\Lambda^p M), \\
&\|e^{-t\Delta}f\|_{x^{\beta} C^{k+n}_e}\leq C t^{-n/2}\|f\|_{x^{\beta} C^k_e}, 
\end{split}
\end{equation}
and with respect to H\"older spaces
\begin{equation}
\label{est:holder-spaces} 
\begin{split}
&e^{-t\Delta}: x^{\beta} C^{\A}_e(M,{}^{e}\Lambda^p M) \to x^{\beta} C^{2}_e(M,{}^{e}\Lambda^p M), \\ 
&\|e^{-t\Delta}f\|_{x^{\beta} C^{2}_e}\leq C t^{-1+\A/2}\|f\|_{x^{\beta} C^{\A}_e}.
\end{split}
\end{equation}
\end{thm}

\noindent Although these estimates may be established
using classical PDE techniques, we emphasize a parallel microlocal approach
to both the incomplete and complete setups.

An immediate observation from these estimates is that after taking into account the geometry of the underlying space, the mapping properties of the heat operator resemble the well-known behaviour on compact manifolds. As a particular consequence of our results, we establish short-time existence for solutions to certain semilinear parabolic equations on manifolds with edge metrics.  Applications of such equations, including the reaction-diffusion equation, may be found in \cite{CazHar:ISE}.

\begin{cor} \label{corr}
Let $(M,g)$ be an $m$-dimensional Riemannian manifold with a feasible edge metric. Let $\Delta$ denote the Friedrichs extension of the Laplace-Beltrami operator on $(M,g)$ for an incomplete edge metric, and the unique self-adjoint extension of the Laplace-Beltrami operator for a complete edge metric. 
Let the pair $(X,Y)$ denote
\begin{enumerate}
 \item either $(C^{2}_e(M), C^{\A}_e(M))$ for $\A\in (0,1)$ or $(\mathscr{C}^{k+1}_e(M), \mathscr{C}^{k}_e(M)), k\in \N$, if $g$ is an incomplete edge metric;
 \item either $(C^{2}_e(M), C^{\A}_e(M))$ for $\A\in (0,1)$ or $(C^{k+1}_e(M), C^{k}_e(M)), k\in \N$, if $g$ is a complete edge metric.
\end{enumerate}
For $u \in X$, let $\Phi:X\to Y$ denote a locally Lipschitz inhomogeneous term. Then the initial value problem
$$\frac{\partial u}{\partial t} \ +\  \Delta \ - \ \Phi(u)=0, \quad u(0)=u_0\in X,$$
admits a solution $u \in C([0,T];X)$ on some time interval $[0,T]$ for $T>0$. 
\end{cor}

\begin{remark}
We have stated Corollary \ref{corr} for the Laplacian acting on functions.  The same result holds in the complete case for the Hodge Laplacian on differential
forms.  At the time of writing we were unable to adapt our argument to the case of the Hodge Laplacian for differential forms in the incomplete setting.  Please see \S \ref{s-short} for further explanation.
\end{remark}

Our discussion of the heat operator in the edge setup reviews and generalizes the work of Jeffres and Loya in \cite{JefLoy:RSH} and \cite{JefLoy:RHM}, where the case $\dim B=0$ was considered. Their work was based on the heat kernel analysis by Mooers \cite{Moo:HMC} in the incomplete (conical) case and by Melrose \cite{Mel:APS} in the complete (b-cylindrical) setup.  
The analysis of the heat operator in the presence of incomplete singularities was initiated by Cheeger \cite{Che:SGS}, with major contributions by Br\"uning and Seeley \cite{BruSee:RES}, \cite{BruSee:ERS}, Lesch \cite{Les:OFT}, Melrose \cite{Mel:APS} and Mazzeo \cite{Maz:ETD}, to select a few. Related questions on regularity properties of solutions to parabolic equations in the singular setup have been studied in \cite{CorSchSei:DOC}, \cite{Li:EON} and \cite{Loy:TOH}.  Most recently, a study of the inhomogeneous Cauchy problem on manifolds with incomplete conical metrics has been presented by Behrndt \cite{Behrndt}.

The paper is organized as follows. In \S\S \ref{s-edges}-\ref{section-algebraic} we review edge operators and the Hodge Laplacian in this setting.  We also discuss algebraic boundary conditions to obtain self-adjoint extensions of the Hodge Laplacian in the incomplete case.  In \S \ref{s-asymptotics} we review the asymptotic properties of the heat kernel as a polyhomogeneous distribution on the appropriate blowup of the heat space.  We discuss the incomplete and complete cases separately since the asymptotic properties are different.  In \S \ref{s-mapping} we apply the asymptotics of the heat kernel to derive the mapping properties of the heat operator and hence the regularity of solutions to the heat equation, carefully estimating the corresponding integral in various regions of the heat-space blowup. Finally, in \S \ref{s-short} we explain how these mapping properties yield short-time existence for solutions to certain semilinear parabolic equations. \bigskip

\emph{Acknowledgements:} 
We are indebted to Rafe Mazzeo for his constant support and interest in this project.
It is also a pleasure for us to acknowledge helpful discussions with Pierre Albin and Paul Loya. 
We thank an anonymous referee for careful reading and important remarks.
The first author appreciates the hospitality of the Bucknell University Mathematics Department under the auspices of the Distinguished Visiting Professor program during the preparation of this paper. The first and second authors are grateful to the Mathematical Sciences Research Institute for giving them the opportunity to be exposed to this area during the Fall 2008 program ``Analysis of Singular Spaces.''  The second author was partially supported by a grant from the Simons Foundation (210445 to Emily B. Dryden), and also thanks the School of Mathematics at Trinity College Dublin for their hospitality.
The third author gratefully acknowledges 
financial support by the German Research Foundation DFG as well as by the Hausdorff Center for 
Mathematics in Bonn, and also thanks Stanford University for its hospitality.

\section{Differential edge operators} \label{s-edges}

We review here some basic elements of  elliptic operator theory on edge spaces, 
as discussed in the seminal paper by Mazzeo \cite{Maz:ETD} (see also \cite{MazVer:ATM}). We consider local coordinates 
$(x,y,z)$ on $\overline{M}$ near the boundary, where $x$ is a boundary defining function, 
coordinates $(y)$ are lifted from the base $B$, and $(z)$ restrict to coordinates on $F$ along each fibre of $\partial M$. 
By definition, a differential edge operator $L \in \textup{Diff}^{n}_e(M)$ can be written locally as a sum of products of 
elements of edge vector fields $\calV_e$, with coefficients in $\calC^\infty(\overline{M})$. Hence
\[
L=\sum_{j+|\A|+|\beta|\leq n} a_{j,\A,\beta}(x,y,z)(x\partial_x)^j(x\partial_y)^{\A}\partial_z^{\beta},
\]
with each $a_{j,\alpha,\beta}$ smooth up to $x=0$. More generally, acting between vector bundle sections, $L$ 
is required to attain this form with respect to suitable local trivializations, with each $a_{j,\alpha,\beta}$ being matrix-valued. 
The operator is called \emph{edge elliptic} if its edge symbol
\[
{}^e \sigma_n(L)(x,y,z;\xi,\eta,\zeta) :=  \sum_{j + |\A| + |\be| = n} a_{j,\A,\be}(x,y,z) \xi^j \eta^{\A} \zeta^{\be}
\]
is nonvanishing (or invertible, if matrix-valued), for $(\xi,\eta,\zeta) \neq (0,0,0)$. 

There are two fundamental operators 
associated to a differential edge operator $L$. To define them, we introduce coordinates $(s,u) \in {\mathbb R}^+ \times {\mathbb R}^b$; this half-space should be interpreted as the inward-pointing normal space to the fibre of $\partial \overline{M}$ through $(0,y_0,z)$.
The first model is the normal operator 
\[
N(L)_{y_0} = \sum_{j + |\A| + |\be| \leq n} a_{j,\A,\be}(0,y_0,z) (s\del_s)^j (s\del_u)^{\A} \del_z^{\be}.
\]
Though almost  of the same complexity as $L$, the normal operator is translation invariant in $u$ and dilation invariant in $(s,u)$ jointly, and is the main ingredient in the construction of both a 
parametrix for $L$ and a heat kernel parametrix for $L$.

The second model is the indicial operator
\[
I(L)_{y_0} = \sum_{j + |\be| \leq n} a_{j,0,\be}(0,y_0,z) (s\del_s)^j \del_z^\be;
\]
Mellin transform in $s$ reduces $I(L)_{y_0}$ to the indicial family
\[
I_\zeta(L)_{y_0} = \sum_{j + |\be| \leq n} a_{j,0,\be} (0,y_0,z) \zeta^j \del_z^\be,
\]
which is a holomorphic family of unbounded Fredholm operators on $L^2(F)$. Values of $\zeta$ for which 
$I_\zeta(L)_{y_0}$ is not invertible are called indicial roots of $L$. It can be shown that the second condition for the edge metric $g$ to be feasible
ensures discreteness of indicial roots in ${\mathbb C}$ for the Hodge Laplacian associated to $g$.

Note that as an operator acting on sections of ${}^{ie} \Lambda^k(M)$,
the $p$-form Hodge Laplacian for an incomplete edge metric $g$ is an element of 
$x^{-2} \textup{Diff}^2_e(M;{}^{ie}\Lambda^k(M))$; for $g$ a complete edge metric, the Hodge Laplacian lies in $\textup{Diff}^2_e(M;{}^e\Lambda^k(M))$.

\section{Hodge Laplacian for incomplete edge metrics}

Let $\Delta_p$ denote the Hodge operator on $p$-forms associated to a feasible incomplete edge metric $g$.
For the convenience of the reader, in this section we review the analysis in \cite{MazVer:ATM} that provides the structure of the normal operator $N(x^2\Delta_p)_{y_0}$ and of its indicial roots.  Let $\mathscr{C}(F)$ denote the cone over $F$, and let $y_0 \in B$ be arbitrary.  Then $N(x^2\Delta_p)_{y_0}$ acts on $p$-forms on the model edge 
$\RR^b \times \mathscr{C}(F) = \RR^+_s \times \RR^b_u \times F$ with incomplete edge metric 
$g_{\textup{ie}} = ds^2 + s^2 g^F + |du|^2$; moreover, the normal operator is naturally identified with 
$s^2$ times the Hodge Laplacian associated to $g_{\textup{ie}}$. 

Consider the tangent space $TS_a \equiv T(\RR^b \times F)$ to a hypersurface $S=\{s=a\}$ of the model edge.  There is an orthogonal decomposition of $TS_a$ as the sum of the tangent space to the $F$ factor (the `vertical' subspace) and the tangent space to the Euclidean factor (the `horizontal' subspace).  This splitting induces a bigrading
\begin{align}\label{bigrading}
\Lambda^p(S_a) = \bigoplus_{j + \ell = p} \Lambda^j(\RR^b) 
\otimes \Lambda^\ell (F) := \bigoplus_{j+\ell = p} \Lambda^{j,\ell}(S_a). 
\end{align}
Let $\Omega^{j,\ell}(S)$ denote the space of sections of the corresponding summand in this bundle decomposition.  
We simplify the description of indicial roots below by writing the normal operator 
$N(x^2\Delta_p)_{y_0}$ on $\RR^b \times \mathscr{C}(F)$ with 
respect to a rescaling of the form bundles; this rescaling was employed by Br\"uning-Seeley \cite{BruSee:ITF}
and in slightly different form in \cite{HM}. More precisely, for each $j,\ell$ with $j + \ell = p$, we 
write ${}^{ie}\Omega^p(\RR^b \times \mathscr{C}(F)):= C^\infty (\RR^b \times \mathscr{C}(F), {}^{ie}
\Lambda^p(\RR^b \times \mathscr{C}(F)))$ and define
\begin{align*}
& \phi_{j,\ell}: \calC^{\infty}(\RR^+, \Omega^{j,\ell-1}(S) \oplus \Omega^{j, \ell}(S)) 
\rightarrow {}^{ie}\Omega^p(\RR^b \times \mathscr{C}(F)), \\
& \qquad (\eta, \mu) \longmapsto s^{\ell-1-f/2}\eta \wedge ds + s^{\ell-f/2}\mu. 
\end{align*}
We denote by $\Phi_p$ the sum of these maps over all $j+\ell = p$. As in the case of 
conical singularities, the resulting transformation
\begin{equation}\label{unitary}
\begin{split}
\Phi_p: L^2\left(\RR^+, L^2\left(\bigoplus_{j+\ell=p} \Omega^{j,\ell-1}(S) 
\oplus \Omega^{j,\ell}(S), g_F(0) + |du|^2 \right) , ds\right) \qquad \\
\longrightarrow L^2(\Omega^p(\RR^b \times \mathscr{C}(F)), g_{\textup{ie}}),
\end{split}
\end{equation}
is an isometry; a calculation yields
\begin{align}\label{laplace}
\Phi_p^{-1}\left[s^{-2}N(x^2\Delta_p)\right]\Phi_p=\left(-\frac{\del^2}{\del s^2}+\frac{1}{s^2}(A-1/4)\right) + \Delta_{\R^b}, 
\end{align}
where $A$ is the nonnegative self-adjoint operator on $\Omega^{\ell-1}(F) \oplus \Omega^{\ell}(F)$ given by 
\begin{align}\label{a}
A=\left(\begin{array}{cc}\Delta_{\ell-1,F} + (\ell-(f +3)/2)^2 & 2(-1)^{\ell}\, \delta_{\ell,F}\\ 2(-1)^{\ell}\, d_{\ell-1,F}& \Delta_{\ell,F}+ 
(\ell-(f - 1)/2)^2\end{array}\right).
\end{align}
As alluded to above, rescaling the form bundles in this way leads to a simple expression for the indicial roots of $\Delta_p$.  Denoting the eigenvalues of $A$ by $\nu_j^2, \nu_j \geq 0$, with corresponding eigenform $\phi_j$, we may write the
corresponding indicial roots of $\Delta_p$ as
\begin{equation}
\gamma_j^{+} = \nu_j + \frac12 \ , \quad \gamma_j^{-} = - \nu_j + \frac12 . 
\label{indroots}
\end{equation}
If $\nu_j >0$, the corresponding solutions have the form $s^{\gamma_j^{\pm}}\phi_j$; for $\nu_j=0$, they are $\sqrt{s}\, \phi_j$ and $\sqrt{s}(\log s) \phi_j$.  Note that for $\nu_j \geq 1$, the solutions of the form $s^{\gamma_j^{-}} \phi_j$ are not in $L^2(\R^+_s)$.

We can define a similar rescaling $\Phi$ using powers of the defining function $x$ 
in each local coordinate chart near the boundary $\partial M$. Setting $x\equiv 1$
away from the collar neighbourhood $U$ of $\partial M$, we can trivially extend this rescaling to the rest of $M$. 
The construction is independent of our choice of coordinate charts, as rescalings on different charts differ by a diffeomorphism.  Thus we may again conjugate by $\Phi$ and consider the rescaled operator $\Delta_p^\Phi$; we will abuse notation and also use $\Delta_p$ to denote the rescaled operator if the meaning is clear from the context.  We may view $\Delta_p$ as a perturbation of \eqref{laplace}, with higher order correction terms determined by the curvature of the Riemannian submersion $\phi: \partial M \to B$ and the second fundamental forms of the fibres $F$.

\section{Algebraic boundary conditions on incomplete edges}\label{section-algebraic}

The Hodge Laplacian of an incomplete edge space $(M,g)$ need 
not be essentially self-adjoint on its core domain $\calC^\infty_0\Omega^p(M)$, 
so we must consider how to impose boundary conditions at the edge to
obtain self-adjoint extensions.  For spaces with isolated conic singularities, this was first
accomplished by Cheeger \cite{Che:SGS}.  Further studies in the conic setting 
appear in \cite{Les:OFT}; see also \cite{Moo:HMC} and \cite{KLP} for results about the associated heat equation. 

In the case of isolated conical singularities the extension problem is finite dimensional. When the edge has positive dimension, 
the requisite analysis is more intricate and here we can specify only a class of \emph{algebraic} boundary conditions 
for the Hodge Laplacian. We begin by introducing conormal and 
polyhomogeneous distributions on a manifold with corners.

\begin{defn}\label{phg}
Let $\mathfrak{W}$ be a manifold with corners, with all boundary faces embedded, and $\{(H_i,\rho_i)\}_{i=1}^N$ an enumeration 
of its boundaries and the corresponding defining functions. For any multi-index $b= (b_1,
\ldots, b_N)\in \C^N$ we write $\rho^b = \rho_1^{b_1} \ldots \rho_N^{b_N}$.  Denote by $\mathcal{V}_b(\mathfrak{W})$ the space
of smooth vector fields on $\mathfrak{W}$ which lie
tangent to all boundary faces. A distribution $w$ on $\mathfrak{W}$ is said to be conormal
if $w\in \rho^b L^\infty(\mathfrak{W})$ for some $b\in \C^N$ and $V_1 \ldots V_\ell w \in \rho^b L^\infty(\mathfrak{W})$
for all $V_j \in \mathcal{V}_b(\mathfrak{W})$ and for every $\ell \geq 0$. An index set 
$E_i = \{(\gamma,p)\} \subset {\mathbb C} \times {\mathbb N}$ 
satisfies the following hypotheses:
\begin{enumerate}
\item $\textup{Re}(\gamma)$ accumulates only at $+ \infty$;
\item for each $\gamma$ there exists $P_{\gamma}\in \N_0$ such 
that $(\gamma,p)\in E_i$ if and only if $p \leq P_\gamma$;
\item if $(\gamma,p) \in E_i$, then $(\gamma+j,p') \in E_i$ for all $j \in {\mathbb N}$ and $0 \leq p' \leq p$. 
\end{enumerate}
An index family $E = (E_1, \ldots, E_N)$ is an $N$-tuple of index sets. 
Finally, we say that a conormal distribution $w$ is polyhomogeneous on $\mathfrak{W}$ 
with index family $E$, denoted $w\in \mathscr{A}_{\textup{phg}}^E(\mathfrak{W})$, 
if $w$ is conormal and if near each $H_i$ we have
\[
w \sim \sum_{(\gamma,p) \in E_i} a_{\gamma,p} \rho_i^{\gamma} (\log \rho_i)^p, \ 
\textup{as} \ \rho_i\to 0,
\]
with coefficients $a_{\gamma,p}$ conormal on $H_i$ and polyhomogeneous with index $E_j$
at any $H_i\cap H_j$. 
\end{defn}

For more on polyhomogeneous distributions and other relevant background material, 
we refer the reader to the classical references \cite{Mel:APS} and \cite{Maz:ETD}, 
as well as the excellent introduction to the $b$-calculus by Grieser \cite{Gr}.

Let $\Delta_p^\Phi = \Delta_p$ denote the rescaled Hodge Laplace operator acting on differential forms of degree $p$
on the incomplete edge space $(M,g)$ with a feasible incomplete edge metric $g$.  Consider the space of $L^2$-forms in $\Omega^p(M)$ 
with respect to $g$, denoted $L^2 \Omega^p(M)$. The \emph{maximal} and \emph{minimal} closed extensions of $\Delta_p$ are defined by the domains
\begin{equation}
\begin{split}
\calD_{\max}(\Delta_p) &:= \{ u\in L^2\Omega^p(M) \mid  \Delta_p u \in L^2\Omega^p(M) \}, \\ 
\calD_{\min}(\Delta_p) &:= \{ u \in \calD_{\max}(\Delta_p)  \mid \exists\,  u_j \in \calC^\infty_0\Omega^p(M)\ \mbox{such that}  \\ 
&u_j \to u\ \mbox{and}\ \Delta_p u_j \to \Delta_p u\ \mbox{, with both sequences converging in}\ L^2\Omega^p(M) \},
\end{split}
\end{equation}
where $\Delta_p u\in L^2\Omega^p(M)$ is initially understood in the distributional sense. 
The set of all closed extensions of $\Delta_p$ is in bijective correspondence with the closed subspaces of the quotient 
$\calD_{\max}/\calD_{\min}$; furthermore, since $\Delta_p$ is symmetric on the core domain 
$\calC^\infty_0\Omega^p(M)$, self-adjoint extensions are in 
bijective correspondence with the subspaces of this quotient which are Lagrangian with respect to a certain natural symplectic 
form induced from the boundary contributions in an integration by parts formula (e.g., \cite[\S 3]{KLP}). 

\begin{lemma} \cite[Lemma 2.2]{MazVer:ATM}\label{max}
Let $(M,g)$ be an incomplete edge space with a feasible edge metric. 
Any $u \in \calD_{\max}(\Delta_p)$ admits a weak asymptotic expansion as $x\to 0$
\begin{equation}\label{expansion-w}
u \, \sim \, \sum_{j=1}^q \left( c_j^+[u](y) \psi_j^+(x,z) + c^-_j[u](y) \psi_j^-(x,z) \right) + 
\tilde{u}, \ \tilde{u} \in \calD_{\min}(\Delta_p),
\end{equation}
where the index $j=1,\dots,q$ counts the indicial roots $\gamma_j^{\pm} = \pm \nu_j + 
\frac{1}{2}$ of $\Delta_p$ with increasing $\nu_j\in [0,1)$ and the leading order term of 
each $\psi_j^{\pm}$ is the corresponding solution of the indicial operator.  More precisely,
$\psi^+_j(x,z)=x^{\gamma_j^+}\phi_j(z)$, where $\phi_j(z)\in \Omega^*(F)$ are
the normalized $\nu_j^2$-eigenforms of the tangential operator \eqref{a} at $y$.
If $\nu_j=0$, then $\psi^-_j(x,z)=\sqrt{x}(\log x)\phi_j(z)$. If $\nu_j>0$, then 
$\psi^-_j(x,z)=x^{\gamma_j^-}(1+a_jx)\phi_j(z)$, with $a_j\in \R$ uniquely
determined by $\Delta_p$.\footnote{The expansion for $\psi_j^-$ is not given explicitly in \cite[Lemma 2.2]{MazVer:ATM}.  The terms $x^{\gamma_j^-+1}$ appear from the standard
argument using the Mellin transform and inverting the holomorphic
indicial family.   These terms are not of sufficiently high order to be absorbed into $\tilde{u}$.}

The coefficients  $c_j^{\pm}[u]$ are of negative regularity in $y$, meaning that these functions are in a Sobolev space over $B$ of negative order. Thus the 
asymptotic expansion holds only in a weak sense; that is, there is an expansion of the pairing
$\int_B (u(x,y,z) \chi(y))_{g^B}\, dy$ for any test function $\chi \in \Omega^*(B)$. 
\end{lemma}

We next give \emph{algebraic} boundary conditions for $\Delta_p$ by specifying algebraic relations among the coefficients $c_j^{\pm}$.
Note that though the discussion above was in the context of rescalings $\Phi_p$ in each local coordinate neighbourhood, the partial weak 
asymptotic expansion of the rescaled solutions in $\calD_{\max}(\Delta)$ is invariant under coordinate changes
$x'(x,y,z), y'(x,y,z)$ and $z'(x,y,z)$, with $x'(0,y,z)=0$ and $y'$ lifted from the base, so that 
$y'(0,y,z)$ is independent of $z$. Thus, any  specification of algebraic relations between the coefficients $c_j^{\pm}$
and hence also the characterization of algebraic boundary conditions for $\Delta_p$ will be globally well-defined.

Following the description of self-adjoint boundary conditions of the Hodge Laplacian on a 
cone in (\cite{Moo:HMC}, Section 7), we consider $\Lambda_q$ the $2q$-dimensional vector space spanned by solutions $\{\psi_j^{\pm}\}_{j=1}^q$.
We define the bilinear form $\w_q$ on $\Lambda_q$ by 
\begin{equation}
\begin{split}
&\w_q(\psi_j^{+},\psi_j^{-}) = - \w_q(\psi_j^{-},\psi_j^{+}) = \left \{
\begin{split} &2\nu_j, \ \nu_j >0, \\ &1, \ \ \ \, \nu_j =0, \end{split} \right. \\
&\w_q(\psi_j^{+},\psi_j^{+}) = \w_q(\psi_j^{-},\psi_j^{-}) = \w_q(\psi_i^{\pm},\psi_j^{\pm}) = 0, i \neq j.
\end{split}
\end{equation}
A subspace of $\Lambda_q$ on which the form $\omega_q$ vanishes can be represented by a $q \times q$ matrix $\Gamma = (\Gamma_{ij})$ with 
diagonal entries $\Gamma_{jj}=b_{jj} \psi_j^{-} + \theta_{jj} \psi_j^{+}$
and off-diagonal entries $\Gamma_{ij}= \theta_{ij} \psi_j^{+}$; the coefficients  
$b_{ij}, \theta_{ij}\in \R$ are such that either $b_{ii}=1$ or $b_{ii}=0$, where in the
latter case we require $\theta_{ii}=1$ and $\theta_{ij}=0$ for $i\neq j$. We refer to such 
a matrix $\Gamma$ as the \emph{Lagrangian matrix}.   If $b_{jj}=0$ whenever $\nu_j=0$, we call $\Gamma$ \emph{non-logarithmic}, as in this case there are no so-called ``unusual'' logarithmic terms in the expansion of the heat trace (cf. \cite{KLP}).

\begin{defn}\label{defn:DG}
For any Lagrangian matrix  $\Gamma = (\Gamma_{ij}),$
we define the associated algebraic domain of the Hodge Laplacian $\Delta_p$ by 
\begin{align*}
\calD_{\Gamma}(\Delta_p) := \{u\in \calD_{\max}(\Delta_p) \mid 
\forall \ i=1,\dots,q: \, \sum_{j=1}^q \w_q( c_j^+[u] \psi_j^+ + c^-_j[u] \psi_j^- , \Gamma_{ij})=0\}.
\end{align*}
\end{defn}

The algebraic boundary conditions provide a full characterization of self-adjoint extensions of the Hodge
Laplacian on cones,  cf. \cite[\S 7]{Moo:HMC} and \cite{KLP}. Note that Prop. 2.5 of \cite{MazVer:ATM} identifies the Friedrichs
extension of $\Delta_p$ as the algebraic self-adjoint extension associated to $\Gamma=\textup{diag}(\psi_1^+,\dots,\psi_q^+)$.
A full classification of self-adjoint extensions of the Hodge Laplacian on incomplete edges requires a detailed analysis 
of the elliptic theory of edge differential operators and is beyond the scope of the present discussion. 

Moreover, standard arguments from the conical setup which show self-adjointness of ${\calD}_{\Gamma}(\Delta_p)$ 
do not apply here directly, since due to the weakness of the asymptotic expansion 
in \eqref{expansion-w} the symplectic form $\w_q$ can be evaluated 
explicitly only on polyhomogeneous $u,v\in \calD_{\max}$ and not on the full $\Lambda_q$.
Hence, even symmetry of $\calD_{\Gamma}(\Delta_p)$ is not obvious here and 
requires a mollification argument.  
The mollification argument, however, does not apply to the second-order degenerate operator $\Delta_p$
in any obvious way unless $B$ is either zero-dimensional or Euclidean. To overcome this technical difficulty,
we write $\Delta = \oplus_p \Delta_p$ for the full Laplacian. Its normal operator is also of the form \eqref{laplace} 
with the tangential operator $A=\oplus_p A_p$. Note that $\Delta = D^t D$, where $D$ denotes 
the Gauss-Bonnet operator of $(M,g)$. By \cite[Lemma 2.4]{MazVer:ATM} any $u \in \calD_{\max}(D)$ 
admits a weak asymptotic expansion as $x\to 0$
\begin{equation}\label{expansion-u}
u \, \sim \, \sum_{j=1}^p c_j[u] \, \phi_j(z;y) \, x^{-\nu_j + 1/2} + \tilde{u}, \ \tilde{u} \in \calD_{\min}(D),
\end{equation}
where $\phi_j$ is some normalized $\nu_j^2$-eigenform of the tangential operator $A$ at $y$,
and the coefficients  $c_j[u]$ are of negative regularity in $y$, i.e. the 
asymptotic expansion holds only in a weak sense. The Lagrange identity for $D$ acting on 
$\calD_{\max}(D) \cap \mathscr{A}_{\textup{phg}}$ is worked out in \cite[(2.9)]{MazVer:ATM}. 
In the conical case, we may classify all self-adjoint extensions of $D$ similar to Definition \ref{defn:DG} by 
specifying linear relations $S$ among the coefficients $c_j[u]$. Each such choice $S$ gives a well-defined domain 
$\calD_{S}(D)$ in the case of incomplete edges and we prove the following  

\begin{prop}
$\calD_{S}(D)$ defines a self-adjoint extension of $D$.
\end{prop}

\begin{proof}
We first prove that $D$ is symmetric on $\calD_{S}(D)$.
Note that we cannot follow the arguments from the conical case directly, since the 
expansion \eqref{expansion-u} holds only in the weak sense.

Consider any solution $w\in \calD_S(D)$.
Let $\phi$ be a cut-off function supported in a local coordinate neighborhood $(x,y,z)$. 
Assume $\phi(x,y,z)=\phi_1(x) \phi_2(y,z)$, where $\phi_1$ is identically one in an open neighborhood 
of $x=0$, and $\phi_2$ a smooth bump function around some $(y_0, z_0)\in \partial M$. 
Then $u:= w \cdot \phi$ is still in $\calD_S(D)$. Consider a coefficient $u_I$ of the form-valued $u$ and a test function $\psi \in C^\infty(B)$ and write 
$$
(u_I * \psi)(x,y,z)= \int_B u_I(x,y-\wy, z) \psi(\wy) d\wy.
$$
We can assemble local functions $u_I * \psi$ back into a differential form, 
which we denote by $u * \psi$. The convolution $u * \psi$ inherits the expansion as $x\to 0$ from $u$, 
and hence $u * \psi \in \calD_S(D)$. Moreover, due to pairing with $\psi \in C^\infty(B)$, the coefficients 
$c_j[u*\psi]$ are now smooth in $y$ and hence $u*\psi \in \calD_S(D) \cap \mathscr{A}_{\textup{phg}}$. 
Specify now $\psi$ to be a bump function, compactly supported around the coordinate origin of $(y)$
with $ \widehat{\psi}(0)=1$, where $\widehat{\psi}$ denotes the Fourier transform of $\psi$; we set $\int_B \psi(y) dy = 1$.  Consider a sequence $\psi_\epsilon (y) := \epsilon^{-b}\psi(y/\epsilon)$. Then
pointwise
$$
\widehat{\psi_\epsilon}(\zeta) = \widehat{\psi}(\epsilon \zeta) \rightarrow \widehat{\psi}(0) =1,
\ \textup{as} \ \epsilon \to 0.
$$
Set $u_\epsilon = u * \psi_\epsilon$. Observe that for each coefficient $u_{\epsilon, I}= u_I * \psi_\epsilon$, the Fourier transform $\widehat{u}_I$ is $L^2$-integrable, and that $(\widehat{\psi}(\epsilon \, \cdot) - 1)$
is bounded uniformly in $\epsilon$; hence by the dominated convergence theorem,
\begin{align}\label{molli}
\| \widehat{u}_{\epsilon, I} - \widehat{u}_{I}\|_{L^2} = \| \widehat{u}_I \left(\widehat{\psi}(\epsilon \, \cdot) - 1\right)\|_{L^2} 
 \rightarrow 0, \ \textup{as} \ \epsilon \to 0.
\end{align}
This proves $u_\epsilon \to u$ in $L^2$ as $\epsilon \to 0$. Moreover, we write componentwise
\begin{align*}
D (u * \psi_\epsilon) &= (D u) * \psi_\epsilon + 
\sum_{k\in \mathscr{K}}  \int_B  \big(a_k(y) - a_k(y-\wy)\big)D_k u(y-\wy) \psi_\epsilon (\wy) d\wy\\
&=: (D u) * \psi_\epsilon + 
\sum_{k\in \mathscr{K}}  \int_B  \delta_{a_k}(y,\wy) D_k u(y-\wy) \psi_\epsilon (\wy) d\wy,
\end{align*}
where $a_k D_k, k \in \mathscr{K},$ is the collection of summands in $D$ with $y$-dependent coefficients, each $a_k\in C^\infty(\overline{M})$, and we incorporate the eventual singular behaviour into $D_k$.
We will show that the sum converges to zero in $L^2$ and hence by exactly the same argument as above, 
$D u_\epsilon \to D u$ in $L^2$ as $\epsilon \to 0$. Thus, we can indeed approximate any 
locally supported $u\in \calD_S(D)$ by
a sequence $(u_\epsilon) \subset \calD_S(D) \cap \mathscr{A}_{\textup{phg}}$ in the graph norm.
\medskip

The conditions that $\partial M$ is a Riemannian submersion over $B$ and $|h|_{g_0}=O(x)$
as $x\to 0$ guarantee that $D_k$ is either in $\mathcal{V}_e$,
or $D_k  \in C^\infty - \textup{span} \{\partial_y\}$. In the first case, by elliptic edge theory (see \cite{Maz:ETD}), $D_ku_I \in L^2$ and hence we may apply the same argument as in \eqref{molli}. The argument for the 
more intricate latter case is different. For some $D_k =\partial_{y_j}$ we use integration by parts to compute (omit the lower index $I$)
\begin{align*}
&\int_B \partial_{y_j} u(y-\wy) \delta_{a_k}(y,\wy)  \psi_\epsilon (\wy) \, d\wy 
= \int_B  u(y-\wy) \,  \partial_{\wy_j}\left(\delta_{a_k}(y,\wy) \psi_\epsilon (\wy)\right) d\wy =\\
&\int_B  u(y-\wy) \,  \partial_{y_j} a_k(y-\wy) \, \psi_\epsilon (\wy) \, d\wy +
\int_B  u(y-\wy) \delta_{a_k}(y,\wy) \partial_{\wy_j} \psi_\epsilon (\wy) \, d\wy =:I_1 + I_2.
\end{align*}

By assumption, both $u$ and $\partial_{y_j} a_k \cdot u$ lie in $L^2$ and hence repeating the 
argument of \eqref{molli} we deduce that $I_1$ converges to $\partial_{y_j} a_k \cdot u$ in $L^2$ as $\epsilon \to 0$. For $I_2$
we expand $a_k(y-\wy)$ in a Taylor series around $y$ and write in the standard multiindex notation
\begin{align*}
I_2 \sim \sum_{|\A|=1}^\infty \frac{(-1)^{|\alpha|+1}}{|\A|!} \, \partial_y^\A a_k(y) \int_B u(y-\wy) \, \wy^\A \,
\partial_{\wy_j} \psi_\epsilon (\wy) \, d\wy
\end{align*}
For $|\A|\geq 2$ the Fourier transform of $\wy^\A \partial_{\wy_j} \psi_\epsilon (\wy)$ has at least one 
additional $\epsilon$ and hence converges pointwise to zero. Thus the corresponding summands 
converge to zero in $L^2$ by repeating the argument of \eqref{molli}. For $|\A|=1$ we let $Y_i$ denote
the operator that multiplies by $\wy_i$ and $\delta_{ij}$ denote the usual Kronecker delta; after integrating by parts, we obtain
\begin{align*}
\lim_{\epsilon \to 0}\widehat{Y_i \partial_{y_j} \psi_\epsilon}(\zeta) = 
\lim_{\epsilon \to 0}\widehat{Y_i \partial_{y_j} \psi}(\epsilon \zeta) =  \int_B \wy_i \, \partial_{\wy_j} \psi (\wy) \, d\wy 
= - \delta_{ij} \int_B \psi (\wy) \, d\wy = -\delta_{ij}.
\end{align*}

Hence, an argument similar to that in \eqref{molli} implies that $I_2$ converges to $(-\partial_{y_j} a_k \cdot u)$ in $L^2$ as $\epsilon \to 0$
and hence $I_1 +I_2$ converges to zero. This proves that we can indeed approximate any 
locally supported $u\in \calD_S(D)$ by
a sequence $(u_\epsilon) \subset \calD_S(D) \cap \mathscr{A}_{\textup{phg}}$ in the graph norm.

We claim that $D$ is symmetric on $\calD_S(D) \cap \mathscr{A}_{\textup{phg}}$.  Let $f,g \in \calD_S(D) \cap \mathscr{A}_{\textup{phg}}$.  Since $f$ is polyhomogeneous, it admits a (strong) asymptotic expansion $f \sim \bar{f} + \tilde{f}$, where $\tilde{f} \in \calD_{\min}(D)$ and $\bar{f}$ is a linear combination of solutions $\{\psi_j^{\pm}\}_{j=1}^q$; similarly, $g \sim \bar{g} + \tilde{g}$.  The same arguments as in \cite[Thm. 7.6]{Moo:HMC} show that 
\[
\langle D f, g \rangle - \langle f, D g \rangle = \omega_q(\bar{f}, \bar{g}).
\]
 But $f$ and $g$ satisfy the boundary conditions determined by $S$, so $\bar{f}$ and $\bar{g}$ are in the Lagrangian subspace associated to $S$ and hence $\omega_q(\bar{f}, \bar{g}) = 0$.

To show symmetry of $D$ on $\calD_{S}(D)$, we consider a partition of unity $(\phi_\A)$ subordinate to coordinate charts of $\overline{M}$.
For any $f,g \in \calD_{S}(D)$ we can write
\begin{align*}
\langle D f, g \rangle_{L^2} - \langle f, D g \rangle_{L^2} = 
\sum_{\A} \left( \langle D f, g \cdot \phi_\A \rangle_{L^2} - \langle f, D ( g\cdot \phi_\A) \rangle_{L^2} \right) = 0, 
\end{align*}
where the last equality follows from the facts that each $g \cdot \phi_\A$ may be approximated by elements of
$\calD_{S}(D) \cap \mathscr{A}_{\textup{phg}}$ in the graph norm, and that $D$
is symmetric on $\calD_{S}(D) \cap \mathscr{A}_{\textup{phg}}$.

In order to deduce self-adjointness on $\calD_{S}(D)$ it now suffices to show
$$\calD(D_{S}^*) := \{f\in \calD_{\max}(D) \mid \forall g \in \calD_{S}(D): 
\langle D f, g \rangle_{L^2} = \langle f, D g \rangle_{L^2} \} \subseteq \calD_{S}(D).$$
Let $f \in \calD(D_{S}^*)$. Then in any coordinate chart we may consider a locally 
supported $g\in \calD_{S}(D) \cap \mathscr{A}_{\textup{phg}}$, with its coefficients $c_j[g]$ 
being smooth with compact support in $\R^b$. The regularity of coefficients in the asymptotic expansion of $f$ 
is no longer an issue due to pairing with polyhomogeneous $g$ and hence, arguing as in the conical case \cite[Prop. 7.4]{Moo:HMC},
we deduce from $\langle D f, g \rangle_{L^2} = \langle f, D g \rangle_{L^2}$ 
that the coefficients $c_j[f]$ in the weak expansion of $f$ in that coordinate neighbourhood must satisfy the algebraic conditions 
of $\calD_{S}(D)$. Consequently, $f \in \calD_{S}(D)$.
\end{proof}

This mollification argument similarly yields self-adjointness of $\calD_{\Gamma}(\Delta_p)$ 
if $B$ is either zero-dimensional or Euclidean. Therefore we henceforth either restrict ourselves 
 to algebraic domains $\calD_{\Gamma}(\Delta_p)$ which arise as restrictions to $L^2\Omega^p$ of 
self-adjoint realizations $D_S^*D_S$, or assume that $B$ is either zero-dimensional or Euclidean.

\section{Asymptotics of the heat kernel on edge manifolds}\label{s-asymptotics}

Recall that we let $(M^m,g)$ be a Riemannian manifold with an edge at $B^b$ and a feasible edge metric $g$, where $M$ is the interior of a compact manifold $\overline{M}$ with boundary.  This boundary $\partial M$ is the total space of a fibration $\phi: \partial M \to B$, with fibre $F^f$. The local coordinates on $\overline{M}$ near $\partial M$ are given by $(x,y,z)$, where $(y)$ are local coordinates on $B$ lifted to $\partial M$ and then extended inwards, and $(z)$ restrict to local coordinates on $F$ along each fibre of $\partial M$.  Note that $m = 1 + b + f$.

Let $\Delta_p$ be a self-adjoint extension of the Hodge Laplacian on $(M,g)$ in the case of an incomplete edge, and the unique self-adjoint extension of the Hodge Laplacian in the case of a complete edge. The heat operator of $\Delta_p$, denoted $e^{-t\Delta_p}$, solves the homogeneous heat problem 
\[ \left\{ \begin{array}{rl} 
&(\partial_t + \Delta_p) u(x,y,z,t)  = 0 \\
&u(x,y,z,0) = \w(x,y,z), \ \w\in \calD(\Delta_p),\end{array} \right. \]
with $u=e^{-t\Delta_p}\w$. It is an integral operator 
\begin{equation} \label{eqn:hk-on-functions}
e^{-t\Delta_p}\w(q) = \int_M H\left( t, q, \widetilde{q} \right) \w(\widetilde{q}) \dv (\widetilde{q}),
\end{equation}
with the heat kernel $H$ being a distribution on the heat space $M^2_h:=\R^+\times \overline{M}^2$,
acting pointwise on ${}^{ie}\Lambda^p_{\widetilde{q}} M$ and taking values in
${}^{ie}\Lambda^p_q M$ in the case of an incomplete 
edge metric $g$. In the case of a complete edge metric, $H(t,q,\widetilde{q}) \in \textup{End} ({}^{e}\Lambda^p_{\widetilde{q}} M, 
{}^{e}\Lambda^p_q M)$. Viewing the heat kernel as a section in either 
${}^{ie} \Lambda^p M \boxtimes {}^{ie} \Lambda^p M$ in the incomplete case, or in 
${}^e \Lambda^p M \boxtimes {}^e \Lambda^p M$ in the complete case, we may also write
\begin{equation} \label{eqn:hk-on-functions2}
e^{-t\Delta_p}\w(q) = \int_M (H\left( t, q, \widetilde{q} \right), \w(\widetilde{q}))_g \dv (\widetilde{q}).
\end{equation}

Consider local coordinates $(t, (x,y,z), (\widetilde{x}, \widetilde{y}, \widetilde{z}))$ in the singular neighbourhood, where $(x,y,z)$ and $(\widetilde{x}, \widetilde{y}, \widetilde{z})$ are coordinates on the two copies of $M$ near the edge. 
The heat kernel $H$ has non-uniform behaviour at the diagonal $D$ and at the submanifold
\begin{align*}
A = \left\{ \begin{array}{ll}
\{ (t, (0,y,z), (0, \wy, \wz))\in \R^+ \times (\Mdel)^2 : t=0, \ y= \wy \}, & \textup{if $M$ is incomplete edge}, \\
\{ (t, (0,y,z), (0, \wy, \wz))\in \R^+ \times (\Mdel)^2 : y= \wy \}, & \textup{if $M$ is complete edge}.
\end{array}\right.
\end{align*}
The asymptotic behaviour of $H$ near these submanifolds of $M^2_h$ depends on the angle of approach to these submanifolds, and is crucial for estimation of the integral \eqref{eqn:hk-on-functions}. This dependence on the angle in the asymptotics of the heat kernel is conveniently treated by introducing polar coordinates around $A$ and $D$. Geometrically this corresponds to appropriate blowups of the heat space $M^2_h$, distinct for the incomplete and complete cases, such that the corresponding heat kernel lifts to a polyhomogeneous distribution on that blowup space, in the sense of Definition \ref{phg}.

\subsection{Heat kernel on spaces with incomplete edge metrics} \label{subsec:kernel_incomplete}
To obtain the correct blowup of $M^2_h$ in the case of an incomplete edge metric, one first does a parabolic blowup of the submanifold $A$ defined above.
The resulting heat space $[M^2_h, A]$ is defined as the union of
$M^2_h\backslash A$ with the interior spherical normal bundle of $A$ in $M^2_h$. The blowup $[M^2_h, A]$ is endowed with the unique minimal differential structure with respect to which smooth functions in the interior of $M^2_h$ and polar coordinates on $M^2_h$ around $A$ are smooth. This blowup introduces four new boundary hypersurfaces, which we denote by ff (the front face), rf (the right face), lf (the left face) and tf (the temporal face).  

The actual heat-space blowup $\mathscr{M}^2_h$ is obtained by a parabolic blowup of $[M^2_h, A]$ along the diagonal
\begin{align*}
D:=\{t=0, x=\wx, y=\wy, z=\wz\}\subset M^2_h,
\end{align*}
lifted to a submanifold of $[M^2_h, A]$. The resulting blowup $\mathscr{M}^2_h$ is defined as before by cutting out the submanifold and replacing it with its spherical normal bundle; a new boundary hypersurface, the temporal diagonal (td), appears.  The blowup $\mathscr{M}^2_h$ is a manifold with boundaries and corners as depicted in \cite[Fig. 3]{MazVer:ATM}.
The appropriate projective coordinates on $\mathscr{M}^2_h$ are given as follows. Near the top corner of ff away from tf the projective coordinates are given by
\[
\rho=\sqrt{t}, \  \xi=\frac{x}{\rho}, \ \widetilde{\xi}=\frac{\widetilde{x}}{\rho}, \ u=\frac{y-\widetilde{y}}{\rho}, \ y, \ z, \ \widetilde{z},
\]
where in these coordinates $\rho, \xi, \widetilde{\xi}$ are the defining functions of the faces ff, rf and lf, respectively. For the bottom corner of ff near lf, the projective coordinates are given by
\[
\tau=\frac{t}{x^2}, \ s=\frac{\wx}{x}, \ u=\frac{y-\widetilde{y}}{x}, \ x, \ y, \ z, \ \widetilde{z},
\]
where in these coordinates $\tau, s, x$ are the defining functions of tf, lf and ff, respectively. For the bottom corner of ff near rf the projective coordinates are obtained by interchanging the roles of $x$ and $\widetilde{x}$. The projective coordinates on $\mathscr{M}^2_h$ near the top of td away from tf are given by 
\[
\eta=\sqrt{\tau}, \ S =\frac{1-s}{\eta},\ U=\frac{u}{\eta}, \  \ Z =\frac{z-\widetilde{z}}{\eta}, \  x, \ y, \ z.
\]
In these coordinates tf is the face in the limit $|(S, U, Z)|\to \infty$, and ff and td are defined by $\widetilde{x}$ and $\eta$, respectively. The blowup heat space $\mathscr{M}^2_h$ is related to the original heat space $M^2_h$ via the obvious `blow-down map' $\beta: \mathscr{M}^2_h\to M^2_h,$
which in local coordinates is simply the coordinate change back to $(t, (x,y,z), (\widetilde{x}, \widetilde{y}, \widetilde{z}))$. 

Consider a local coordinate chart $\mathscr{U}_{(x,y,z)}$ and note that as in \eqref{bigrading}, the tangent space to
its hypersurface $\mathscr{U}_{x_0} =\{x=x_0\} \cap \mathscr{U}$ 
splits into the sum of a `vertical' and `horizontal' subspace, where the first is the 
tangent space to the fiber $F_z$-factor and the second is the tangent space to the base $B_y$-factor. 
This splitting is orthogonal, and induces a bigrading 
\begin{equation}\label{eqn:bigrading}
\Lambda^p(\mathscr{U}_{x}) = \bigoplus_{\ell + k= p} \Lambda^\ell(B_y) 
\otimes \Lambda^k (F_z) := \bigoplus_{\ell +k = p}  \Lambda^{\ell,k}(\mathscr{U}). 
\end{equation}
Then, locally, under the rescaling transformation $\Phi$ in \eqref{unitary}, the
heat kernel is a distribution on $\R^+ \times \mathscr{U}^2$, acting on and taking values 
in  
$$ \bigoplus_{\ell + k= p} \left(  \Lambda^{\ell,k-1}(\mathscr{U}) \oplus 
 \Lambda^{\ell,k}(\mathscr{U})\right). $$

We examine the asymptotic properties of the heat kernel as a polyhomogeneous distribution on the blowup $\mathscr{M}^2_h$. 
Denote by $\Delta_\Gamma$ the self-adjoint extension of the $p$-form Hodge Laplacian on the feasible incomplete edge space $(M,g)$, 
with algebraic boundary conditions defined by the Lagrangian matrix $\Gamma$. The Friedrichs extension is denoted by $\Delta_{\mathscr{F}}$. 
The focus of \cite{Moo:HMC} is the analysis of the asymptotic properties of the heat kernel for 
the Hodge Laplacian in the setting of isolated conical singularities, with general self-adjoint boundary conditions at the cone tip. The corresponding 
heat kernel is constructed by adding correction terms to the Friedrichs heat kernel.  In the setting of feasible incomplete edge metrics, arguments like those given by Mooers lead to the following result for the Friedrichs heat kernel.

\begin{thm}\cite{MazVer:ATM}\label{thm-Fried}
The rescaled heat kernel $\HF$ of $\DF$ lifts via the blowdown map $\beta$ to a polyhomogeneous distribution $\beta^*\HF$ on $\mathscr{M}^2_h$, with asymptotic expansion of leading order $(-1-\dim B)$ at the front face ff and leading order $(-m)$ at the diagonal face td, with index sets at the right and left boundary faces given by the indicial roots $\gamma = \nu_i + \frac{1}{2}, \, i \in \N_0,$ of the Hodge Laplacian.
\end{thm}

In fact the leading order of the heat kernel asymptotics at the front face corresponds to the homogeneity order $(-1-b)$
of the delta function on differential forms acting by \eqref{eqn:hk-on-functions2} on differential forms 
with respect to the rescaling $\Phi$.

Mooers' arguments \cite[Thm. 8.2]{Moo:HMC} can then be applied to give the following result for the heat kernel of $\DG$; see \cite[Thm. 8.4]{Ver:HTE} for more details.

\begin{thm}\label{thm-incomplete}
Consider non-logarithmic algebraic boundary conditions $\Gamma$, such that $b_{jj}=0$ for $i=1,\dots,(q'-1)\leq q$.
In particular we require $b_{jj}=0$ whenever $\nu_j=0$. The associated (rescaled) heat kernel $\HG$ 
may be written as a sum
$$\HG = \HF + \sum_{i,j = q'}^q E^{ij}_\Gamma,$$
where each summand $E^{ij}_\Gamma$ lifts via the blowdown map $\beta$ to a polyhomogeneous distribution 
$\beta^*E^{ij}_\Gamma$ on $\mathscr{M}^2_h$, with asymptotic expansion of leading order 
$(-1-\dim B + \nu_i + \nu_j)$ at the front face, leading orders $(-\nu_i +1/2)$ and $(-\nu_j + 1/2)$ 
at the left and right boundary faces, respectively, and vanishing to infinite order at $t\!f$ and $td$.
\end{thm}

Note that Mooers \cite[Thm. 8.2]{Moo:HMC} asserts that for $\nu_i, \nu_j \neq 0$, the correction terms $E_\Gamma^{ij}$ are elements of the space of polyhomogeneous distributions on $\mathscr{M}^2_h$ vanishing to infinite order at td and of leading order $-1 + \nu_i + \nu_j$ at ff.  See \cite{MazVer:ATM} for an explanation of the leading order given above in the presence of an edge $B^b$.  The leading orders of $E_\Gamma^{ij}$ at lf and rf are consequences of the boundary conditions $\Gamma$.  Mooers' Theorem 8.2 allows $\nu_i = 0$ or $\nu_j = 0$, in which case the correction terms $E_\Gamma^{ij}$ are not polyhomogeneous; we exclude these possibilities and consider only non-logarithmic algebraic boundary conditions.

We will also need a more refined result on the asymptotic behaviour of the Friedrichs 
heat kernel on functions at the right boundary face of $\mathscr{M}^2_h$. 
In the function setting the rescaling transformation $\Phi$ simply shifts the indicial roots and 
the heat kernel asymptotics at the right and left boundary faces of $\mathscr{M}^2_h$.
Hence for $p=0$ we do not need to rescale the operators and consider the heat kernel asymptotics
without the transformation $\Phi$.

\begin{prop}\label{heat-expansion}
Let $(M^m,g)$ be a feasible edge space.
The lift $\beta^*\HF$ is a polyhomogeneous distribution on $\mathscr{M}^2_h$; letting $s$ denote a defining function of rf, 
the expansion of $\beta^*\HF$ takes the form
$$
\beta^*\HF \sim s^0 a_{0}(\beta^*\HF) + O(s^{\gamma_0})
 \text{  as  } s \to 0
$$
for some $\gamma_0 >0$, where $a_{0}(\beta^* \HF)$ is constant in the first fibre variable $z$.
\end{prop}

\begin{proof}
We follow the heat kernel construction in \cite{MazVer:ATM}.  Note that $\Delta = \DF$ throughout the proof.
The initial approximate parametrix for the solution operator of $\mathcal{L}=(\partial_t+\Delta)$ 
is constructed by solving the heat equation to 
first order at the front face ff of $\mathscr{M}^2_h$. 
We actually work with the operator $t\mathcal{L}$ rather than just $\mathcal{L}$, as $t\mathcal{L}$ lifts to an operator which is smooth at ff.  Since $\mathcal{L}\HF = 0$ implies $t\mathcal{L}\HF = 0$ and so forth, this change does not affect the outcome of the construction.
The restriction of the lift $\beta^*(t\mathcal{L})$ to ff is called the normal operator 
$N_{\ff}(t\mathcal{L})$ at the front face and feasibility of $g$, more precisely Definition 
\ref{def-feasible} (i), ensures that $N_{\ff}(t\mathcal{L})$ is given in projective coordinates $(\tau,s,u,z,\wx,\wy,\wz)$ by
\[
\beta^*(t\mathcal{L}) |_{\ff} = N_{\ff}(t\mathcal{L})=\tau (\partial_{\tau} - \partial_s^2 - fs^{-1}\partial_s +s^{-2}\Delta_{F,z} + \Delta^{\R^b}_u) \\
=:\tau (\partial_{\tau} + \Delta^{\mathscr{C}(F)}_s + \Delta^{\R^b}_u).
\]
Here $\mathscr{C}(F)=\R^+_s\times F_z$ denotes the model cone.  Since $N_{\ff}(t\mathcal{L})$ 
does not involve derivatives with respect to $(\wx, \wy,\wz)$, it acts 
tangentially to the fibres of the front face. Searching for an initial parametrix $H_0$, we solve 
the heat equation to first order at the front face; we want 
\[
N_{\ff}(t\mathcal{L}\circ H_0)= N_{\ff}(t\mathcal{L}) \circ N_{\ff}(H_0) = 0,
\]
which is the heat equation on the model edge $\mathscr{C}(F)\times \R^{b}_{u}$.
Consequently, the initial parametrix $H_0$ is defined by choosing $N_{\ff}(H_0)$ to 
equal the fundamental solution for the heat operator $N_{\ff}(t\mathcal{L})$, and extending 
$N_{\ff}(H_0)$ trivially to a neighbourhood of the front face.
Using the projective coordinates $(\tau,s,u,z,\wx,\wy,\wz)$ near $\ff$ we have
\begin{equation}\label{normal-heat-ops}
N_{\ff}(H_0) := H^{\mathscr{C}(F)}(\tau,s,z,1,\wz) H^{\R^b}(\tau, u ,0),
\end{equation}
where $H^{\R^b}$ denotes the euclidean heat kernel on $\R^b$ and $H^{\mathscr{C}(F)}$ is the heat kernel 
for the model cone $\mathscr{C}(F)$, as studied in \cite{Che:SGS}, \cite{Les:OFT} and \cite{Moo:HMC}. 
Thus, the index set $E \subset [0, \infty) \times \{0\}$ for the asymptotic behaviour of $H_0$ at the right and left boundary faces 
is given in terms of the indicial roots $\gamma$ for the Laplacian $\Delta^{\mathscr{C}(F)}_s$ on the exact cone, i.e.,
\begin{align*}
H_0 \sim \sum_{\gamma \in E} s^{\gamma} a_{\gamma}(H_0) \text{  as  } s\to 0,
\end{align*}
with the leading coefficient $a_0(H_0)$ being harmonic on fibres and hence constant in $z$. 
Feasibility of $g$, more precisely Definition \ref{def-feasible} (ii), ensures that the indicial 
roots are independent of the base point $b\in Y$; hence the index set $E$ is discrete and 
$H_0$ is indeed polyhomogeneous on $\mathscr{M}^2_h$.
The error of the initial parametrix $H_0$ is given by 
\begin{align*}
\beta^*(t\mathcal{L})H_0=\left[\beta^*(t\Delta_g) - \tau \Delta^{\mathscr{C}(F)}_s - \tau \Delta^{\R^b}_u\right]H_0=:P_0.
\end{align*}
The asymptotic expansion of $H_0$ as $s\to 0$ starts with 
$s^0a_0$ with no $s^0(\log s)^k, k\geq 1$ terms. The leading term $s^0a_0$ is annihilated by 
the edge vector fields $s\partial_s$ and $\partial_z$ since $a_0$ is constant in $z$, and its order 
is raised by $s\partial_u$. Consequently, since $\left[\beta^*(t\Delta_g) - \tau \Delta^{\mathscr{C}(F)}_s - \tau \Delta^{\R^b}_u\right]$ 
is of higher order in $s$ than $N_{\ff}(t\mathcal{L})$, we find
\begin{align*}
P_0 \sim \sum_{l=0}^{\infty}s^l a_{l}(P_0) + 
\sum_{\gamma \in E^*} \sum_{l=0}^{\infty}
s^{\gamma-1+l} a_{\gamma,l}(P_0) \text{  as  }  s\to 0,
\end{align*}
where $E^*:=E\backslash\{0\}$.

The next step in the construction of the heat kernel involves adding a kernel $J$ to $H_0$, 
such that the new error term is vanishing to infinite order at rf. 
In order to eliminate the term $s^\gamma a_\gamma$ in the asymptotic 
expansion of $P_0$ at rf, we only need to solve 
\begin{align}\label{indicial}
(-\partial_s^2- fs^{-1}\partial_s + s^{-2}\Delta_{F,z}) \omega = s^\gamma (\tau^{-1}a_\gamma).
\end{align}
This is because all other terms in the expansion of $t\mathcal{L}$ at rf lower the 
exponent in $s$ by at most one, while the indicial part lowers the exponent by two. 
The variables $(\tau, u, \wx, \wy, \wz)$ enter the equation only as parameters. 
The equation on the model cone is solved via the Mellin tranform and the solution $\omega$ is polyhomogeneous in 
all variables including parameters, and is of leading order $(\gamma+2)$. 
Hence $H_1=H_0+J$ expands near rf as
$$
H_1 \sim s^0 a_0(H_1) + \sum_{l=2}^{\infty}s^la_l(H_1) + 
\sum_{\gamma \in E^*} \sum_{l=0}^{\infty}
s^{\gamma+l} a_{\gamma,l}(H_1) \text{  as  } s\to 0,
$$
where $a_{0}(H_1)\equiv a_0(H_0)$ is constant in $z$.

In the following correction steps the exact heat kernel is 
obtained from $H_1$ first by improving the error near td and then by an iterative correction procedure, 
adding terms of the form $H_1\circ (P_1)^k$, where $P_1:=t\mathcal{L}H_1$ 
is vanishing to infinite order at rf and td. 
Since the leading coefficient $a_{0}(H_1)$ is constant in $z$, we deduce $\partial_z H_1 \sim O(s^{\gamma_0})$ as $s\to 0$, 
where $\gamma_0:=\min \{2, \gamma \in E^*\}$. 
Consequently, we also have $\partial_z H_1\circ (P_1)^k \sim O(s^{\gamma_0})$ as $s\to 0$. We find

\[
\beta^* \HF \sim s^0 a_{0}(\beta^* \HF) + \sum_{l=2}^{\infty} s^la_l(\beta^* \HF)+
\sum_{\gamma \in E^*} \sum_{l=0}^{\infty}s^{\gamma+l} a_{\gamma,l}(\beta^* \HF)
\text{  as  } s \to 0,
\]
where the leading coefficient $a_{0}(\beta^* \HF)\equiv a_0(H_0)$ is still constant in $z$.
\end{proof}

\subsection{Heat kernel on spaces with a complete edge metric} 

We next describe the blowup space and asymptotic structure for the heat kernel of the Laplacian of a complete edge metric.  The blowup proceeds exactly as in \S \ref{subsec:kernel_incomplete} except that in the first step we blow up the heat space $M^2_h$ along the submanifold
\[ A = \{ (t, (0,y,z), (0, \wy, \wz))\in \R^+ \times (\Mdel)^2: y= \wy \} \subset M^2_h. \]
 The final heat-space blowup $\mathscr{M}^2_h$ for a complete edge metric is depicted in Figure \ref{fig:heat-complete}.

\begin{figure}[h]
\begin{center}
\begin{tikzpicture}

\draw[very thick] (-0.7,-0.5) -- (-2,-1);
\draw[very thick] (0.7,-0.5) -- (2,-1);

\draw[very thick] (0.7,-0.5) -- (0.7,1.8);
\draw[very thick] (-0.7,-0.5) -- (-0.7,1.8);

\draw[very thick] (-0.7,-0.5) .. controls (-0.5,-0.6) and (-0.4,-0.7) .. (-0.3,-0.7);
\draw[very thick] (0.7,-0.5) .. controls (0.5,-0.6) and (0.4,-0.7) .. (0.3,-0.7);

\draw[very thick] (-0.3,-0.7) .. controls (-0.3,-0.3) and (0.3,-0.3) .. (0.3,-0.7);
\draw[very thick] (-0.3,-1.4) .. controls (-0.3,-1) and (0.3,-1) .. (0.3,-1.4);

\draw[very thick] (0.3,-0.7) -- (0.3,-1.4);
\draw[very thick] (-0.3,-0.7) -- (-0.3,-1.4);

\node at (1.2,0.7) {\large{rf}};
\node at (-1.2,0.7) {\large{lf}};
\node at (1.1, -1.2) {\large{tf}};
\node at (-1.1, -1.2) {\large{tf}};
\node at (0, -1.8) {\large{td}};
\node at (0,0.3) {\large{ff}};
\end{tikzpicture}
\end{center}
\caption{Heat-space blowup $\mathscr{M}^2_h$ for complete edge metrics}
\label{fig:heat-complete}
\end{figure}
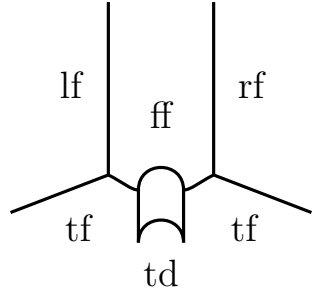\bigskip

We now describe some coordinates on $\mathscr{M}^2_h$.  Near the left hand corner where the ff, tf and lf meet we may use projective coordinates for $\mathscr{M}^2_h$, valid away from $x = 0$:

\begin{align}\label{left-coord}
\tau=\sqrt{t}, \  s=\frac{\wx}{x}, \ u=\frac{y-\wy}{x}, \ x, \ y, \ z, \ \wz.
\end{align}
Observe that in these coordinates, $s$ and $x$ are defining functions for lf and ff, respectively.
Near the top of the blown up diagonal, where td and ff meet, and away from tf we may use the coordinates
\begin{align}\label{diag-coord}
\eta=\sqrt{t}, \  S=\frac{x-\wx}{x \sqrt{t} }, \ U=\frac{y-\wy}{x\sqrt{t}}, \ Z = \frac{z-\wz}{\sqrt{t}}, \ x, \ y, \ z.
\end{align}
In these coordinates, $\eta$ is a defining function for td and $x$ is a defining function for ff.
By abuse of notation, we will denote the blowdown map with respect to any coordinates by $\beta$, and will denote the composition of $\beta$ with projections onto the left and right factors in $M^2_h$ by $\beta_L: \mathscr{M}^2_h\to \bR^+ \times M$ and $\beta_R: \mathscr{M}^2_h\to  M$, respectively.

The asymptotic structure of the heat kernel of the Laplacian of a complete edge metric,
as a section of ${}^e \Lambda^p M \boxtimes {}^e \Lambda^p M$ was described by Pierre Albin in \cite{Alb:RIT}.  
\begin{thm}\cite{Alb:RIT} \label{thm:Albin-htkrnl}
The heat kernel $H$ lifts via the blowdown map $\beta$ to a polyhomogeneous distribution $\beta^*H$ on $\mathscr{M}^2_h$ with asymptotic expansion smooth up to the boundary of  the front face ff, with leading order $(-\dim M)$ at the temporal diagonal face td, and vanishing to infinite order at left, right and temporal boundary faces, lf, rf, tf, respectively.
\end{thm}

The construction of the heat kernel in \cite{Alb:RIT} is parallel for the Hodge Laplacian on functions and on differential forms.
The asymptotics remain the same for any degree
when the heat kernel is viewed as a section of ${}^e \Lambda^p M \boxtimes {}^e \Lambda^p M$.  In fact, the leading order of the heat kernel asymptotics at the front face corresponds to the homogeneity order $(-m)$
of the delta function on differential forms acting by \eqref{eqn:hk-on-functions2} on differential forms.

Note that \cite{Alb:RIT} works in the context of the heat operator acting on half-densities. We dispense with half-densities and regard the heat kernel as acting on differential forms,
which accounts for the difference in powers of the defining functions between Theorem \ref{thm:Albin-htkrnl}
and \cite{Alb:RIT}. We sketch how to apply Albin's result to get the asymptotics of the heat kernel as given above.

Differential forms on $M,\R^+ \times M$ and $\R^+ \times M^2$, 
smooth under the form bundle trivialization \eqref{triv}, may be identified with half-densities 
as follows \footnote{These densities have a factor of $\sqrt{\det g}$ in appropriate places.  In what follows we omit smooth factors.}
\begin{align*}
 \w (t,(x,y,z),(\wx,\wy,\wz)) &\leftrightarrow \w (t, (x,y,z),(\wx,\wy,\wz)) x^{-\frac{(b+1)}{2}} (\wx)^{-\frac{(b+1)}{2}} |dx \,dy \,dz \,d\wx \,d\wy \,d\wz \,dt|^{1/2}, \\
\w (t, (x,y,z)) &\leftrightarrow \w (t, (x,y,z)) x^{-\frac{(b+1)}{2}} |dx \,dy \,dz \,dt|^{1/2}, \\
  \w (x,y,z) &\leftrightarrow \w (x,y,z) x^{-\frac{(b+1)}{2}} |dx \,dy \,dz|^{1/2}. 
\end{align*}

As discussed in \cite{Alb:RIT}, an element of the heat calculus $A \in \Psi^{2,0}_{e, Heat}$ has an integral kernel that may be written as $\rho_{td}^{-\frac{(b+f)}{2}} \rho_{\textup{ff}}^{-\frac{(b+1)}{2}} k \cdot \nu$, where $\rho_{\text{face}}$ is a defining function for the specified face; $k$ is a function that vanishes to infinite order at lf, rf, and tf, and is smooth up to the boundary at td and ff; and $\nu$ is a smooth section of $\Omega^{1/2}(\mathscr{M}^2_h)$.
Such an operator $A \in \Psi^{2,0}_{e, Heat}$ acts on half-densities by
\begin{align} \label{eqn:actionhalf} 
A \; \w(t,(x,y,z)) & x^{-\frac{(b+1)}{2}}   |dx \, dy \,dz \,dt|^{1/2} \nonumber \\
&= (\beta_L)_*\left( \rho_{td}^{-\frac{(b+f)}{2}} \rho_{\textup{ff}}^{-\frac{(b+1)}{2}} k \nu \cdot (\beta_R)^*( \w (\wx,\wy,\wz) (\wx)^{-\frac{(b+1)}{2}} |d\wx \,d\wy \,d\wz|^{1/2} ) \right). 
\end{align}

To relate these half-densities, let us work in coordinates \eqref{diag-coord}.  We may take $\nu = |dS \,dU \,dZ \,dx \,dy \,dz \,d\eta|^{1/2}$.  Pulling back our standard half-density on $M^2 \times \bR^+$, we find
\begin{align} 
\beta^* &( x^{-\frac{(b+1)}{2}} (\wx)^{-\frac{(b+1)}{2}} |dx \,dy \,dz \,d\wx \,d\wy \,d\wz \,dt|^{1/2} ) \nonumber \\ 
&= x^{-\frac{(b+1)}{2}} (x(1-\eta S))^{-\frac{(b+1)}{2}}(-1)^{\frac{b+f+1}{2}} \sqrt{2} x^{\frac{b+1}{2}} \eta^{\frac{b+f+2}{2}} |dS \, dU \, dZ \,dx \,dy \,dz \,d\eta|^{1/2} \nonumber \\ 
&= \sqrt{2} (-1)^{\frac{m}{2}}(1-\eta S)^{-\frac{(b+1)}{2}} x^{-\frac{(b+1)}{2}} \eta^{\frac{b+f+2}{2}} \nu. \label{eqn:pullback_half_den}
\end{align}
Since $1 - \eta S = \tilde{x}/x$ and we are working near the intersection of td and ff, the factor $\sqrt{2} (-1)^{\frac{m}{2}}(1-\eta S)^{-\frac{(b+1)}{2}}$ is smooth and uniformly bounded, so we absorb it into $k$.  
Denote $\sigma_R := \beta_R^* ( (\wx)^{-\frac{(b+1)}{2}} |d\wx \,d\wy \,d\wz|^{1/2} )$.  Using \eqref{eqn:actionhalf} and \eqref{eqn:pullback_half_den} we compute
\begin{align*}
& [A \; \w](t,(x,y,z)) x^{-\frac{(b+1)}{2}} |dx \,dy \,dz \,dt|^{1/2} \nonumber \\
&= (\beta_L)_*\left( \eta^{-\frac{(b+f)}{2}} x^{-\frac{(b+1)}{2}} k \nu \cdot (\beta_R)^*( \w (\wx,\wy,\wz) (\wx)^{-\frac{(b+1)}{2}} |d\wx \,d\wy \,d\wz|^{1/2} ) \right) \\
&= (\beta_L)_*\left( \eta^{-\frac{(b+f)}{2}} x^{-\frac{(b+1)}{2}} k(\eta,(S,U,Z),(x,y,z))  \right. \\ 
& \hspace{ 1 in} \left. \cdot x^{\frac{(b+1)}{2}} \eta^{-\frac{(b+f+2)}{2}} \beta^*( x^{-\frac{(b+1)}{2}} (\wx)^{-\frac{(b+1)}{2}} |dx \,dy \,dz \,d\wx \,d\wy \,d\wz \,dt|^{1/2} )  (\beta_R^* \w) \sigma_R \right) \\
&= (\beta_L)_*\left( \eta^{-m} k(\eta, (S,U,Z),(x,y,z)) (\beta_R^* \w)  \right. \\ 
& \hspace{ 1 in} \left. \cdot \, \beta^*( x^{-\frac{(b+1)}{2}} (\wx)^{-\frac{(b+1)}{2}} |dx \,dy \,dz \,d\wx \,d\wy \,d\wz \,dt|^{1/2} ) \sigma_R \right) \\
&= \left[ \int \eta^{-m}k\left( t^{1/2}, \frac{x-\wx}{x t^{1/2}}, \frac{y-\wy}{x t^{1/2}}, \frac{z-\wz}{t^{1/2}}, x, y, z\right) \w(\wx,\wy,\wz) \wx^{-(b+1)} d\wx d\wy d\wz \right]  \\
& \hspace{ 1 in} \cdot x^{-\frac{(b+1)}{2}} |dx dy dz dt|^{1/2}. 
\end{align*}
Cancelling the half density from the first and last components of this equation, we see that the heat kernel $\beta^* H$ may be expressed as $\eta^{-m}k$, which has the stated asymptotics.

\section{Mapping properties of the heat kernel on edge manifolds}\label{s-mapping}

\subsection{Estimates on spaces with incomplete edge metrics} \label{s:est-incomplete}

We prove Theorem \ref{thm:main:incomplete}.  The main task in the proof is to check the estimates, as the continuity of $e^{-t \Delta_{\Gamma}} f$ and its derivatives are easy to obtain by the dominated convergence theorem in each of the coordinate systems that will be discussed.  The estimates in Theorem \ref{thm:main:incomplete} are classical for the lift of the heat operator to $\mathcal{M}_h^2$ supported away from the front face. So we can assume without loss of generality that the lift of the heat kernel to the heat space blowup $\mathscr{M}^2_h$ is compactly supported in an open neighbourhood of the front face. Then the heat operator acts as in \eqref{eqn:hk-on-functions} and the proof amounts to estimating $(k+n)$ edge derivatives $V_e^{k+n}$ locally.

Consider a local coordinate chart $\mathscr{U}_{(x,y,z)}$ and recall the bigrading of $\Lambda^p(\mathscr{U}_{x})$ given in \eqref{eqn:bigrading}.  For $k \in \N$, 
define the local spaces of continuous $k$-times edge-differentiable functions
\begin{align*}
C^k_e(\mathscr{U}, \Lambda^p T^*\mathscr{U})^\Phi:= 
C^k_e \left(\mathscr{U}, \bigoplus_{l+k=p} \Lambda^{\ell,k-1}(\mathscr{U}) \otimes  \Lambda^{\ell,k}(\mathscr{U})
\right).
\end{align*}
More precisely, consider a frame $\{\theta_1,\dots,\theta_{j_p}\}$ of $\bigoplus_{l+k=p} \Lambda^{\ell,k-1}(\mathscr{U}) \otimes  \Lambda^{\ell,k}(\mathscr{U})$, orthonormal with respect to $dy^2 + dz^2$. Consider $\w= \sum \w_i \theta_i$ with $\w_i \in C(\overline{M})$. Then we say that $\w \in C^k_e(\mathscr{U}, \Lambda^p T^*\mathscr{U})^\Phi$ if, for any choice of edge vector fields $V_j \in \mathcal{V}, j \leq k$, each $V_1\dots V_j \, \w_i\in C(\overline{M})$.
We write 
$$V_1 \dots V_j \, \w := \sum_{i=1}^{j_p} (V_1 \dots V_j\w_i) \, \theta_i.$$
We set
$\|\w\|= \sup_{i} \|\w_i\|_\infty,$
and define a norm on $C^k_e(\mathscr{U}, \Lambda^p T^*\mathscr{U})^\Phi$ by
\begin{align}\label{norm-k}
\| \w \|_k = \| \w \| + \sum_{j \leq k}  \| V_1 \dots V_j \w \|.
\end{align}
Though a priori this norm is not coordinate invariant, coordinate changes lead to equivalent norms.
The H\"older space $C^{\A}_e(\mathscr{U}, \Lambda^p T^*\mathscr{U})^\Phi, \A\in (0,1),$ 
 consists of $\w= \sum \w_i \theta_i$ with $\w_i \in C(\overline{M})$ such that the H\"older norm
\begin{align}\label{norm-a}
\| \w \|_\A = \| \w \| + 
\sup_{i\in \{1,\dots,j_p\}}  \left\|\frac{\w_i(q) -  \w_i(q')}{d(q,q')^\A}\right\|_{\infty, (q,q')\in \overline{M}^2},
\end{align}
is finite, where $d(q,q')$ represents the distance with respect to the incomplete edge metric $g$.
Continuity with respect to the topology on $M$ induced by the incomplete edge metric $g$ 
defines spaces $\mathscr{C}^{*}_e(\mathscr{U}, \Lambda^p T^*\mathscr{U})^\Phi$ in an analogous way,
as Banach subspaces of $C^{*}_e(\mathscr{U}, \Lambda^p T^*\mathscr{U})^\Phi$.

The local rescaling $\Phi$ is, by construction of the norms, an isometry between 
either $C^{k}_e(\mathscr{U}, \Lambda^p T^*\mathscr{U})^\Phi$ and $x^{-f/2}C^k_e(M,{}^{ie}\Lambda^p M)$, or 
between $C^{\A}_e(\mathscr{U}, \Lambda^p T^*\mathscr{U})^\Phi$ and $x^{-f/2}C^\A_e(M,{}^{ie}\Lambda^p M)$.
Consequently, we may establish local estimates for the rescaled heat kernel with respect to the 
ambient local norms \eqref{norm-k} and \eqref{norm-a}. We want to establish boundedness of the rescaled heat operator 
with respect to weighted spaces  $x^{\beta}C^k_e(\mathscr{U}, \Lambda^p T^*\mathscr{U})^\Phi$ and $x^{\beta}C^\A_e(\mathscr{U}, \Lambda^p T^*\mathscr{U})^\Phi$, so we study the integrals 

\begin{eqnarray}
\nonumber && \int_{\mathscr{U}} V_e^{k+n} x^{-\beta} \HG\left( t, (x,y,z), (\wx, \wy, \wz) \right) \wx^{\beta}  \w(\wx,\wy,\wz) d\wx \, d\wy \,d\wz  \\ 
&&= \int_{\mathscr{U}} V_e^{k+n} (x/\wx)^{-\beta} \HG\left( t, (x,y,z), (\wx, \wy, \wz) \right) \w(\wx,\wy,\wz)  d\wx \dvb(\wx), \label{VH-incomp}
\end{eqnarray}
with either $\w \in C^k_e(\mathscr{U}, \Lambda^p T^*\mathscr{U})^\Phi$ or $\w \in C^\A_e(\mathscr{U}, \Lambda^p T^*\mathscr{U})^\Phi$.
For non-Friedrichs boundary conditions $\Gamma$, we specify the weight range by $\beta \in (-3/2 + \nu_q, 1/2 -\nu_q]$ to ensure integrability of the expressions in \eqref{VH-incomp}. If $\Gamma = \mathscr{F}$, we specify the range of weights by $\beta \in (-3/2 - \nu_{\min}, 1/2 +\nu_{\min}]$, where $\nu_{\min} \geq 0$ is the smallest number such that $(\nu_{\min} +1/2)$ is an indicial root of $\Delta$.
Note that if $\nu_{\min} \in [0,1)$, then $\nu_{\min} =\nu_0$.
These weight ranges differ from those given in Theorem \ref{thm:main:incomplete} by $-\frac{f}{2}$ since we are working in the rescaled setting.

The estimates require a precise understanding of the asymptotic behaviour of $\HG$. Since the asymptotics of the heat kernel $\HG$ near the front face are non-uniform, one needs to estimate the integral \eqref{VH-incomp} for the integrand supported near the various corners and boundary faces of the heat space blowup $\mathscr{M}^2_h$. We thus separate $\HG$ into four components, with their lifts to the heat space blowup $\mathscr{M}^2_h$ compactly supported near each of the three corners of the front face and near the diagonal of  $\mathscr{M}^2_h$. This corresponds to different asymptotic regimes as the variables $(t,x,\wx)$ approach zero, with the associated projective coordinates from 
\S \ref{subsec:kernel_incomplete} capturing the various asymptotic behaviour. We establish the estimates for the integral operators defined by each of the components separately. \bigskip

\underline{\emph{Estimates near the lower left corner of the front face}} \ \bigskip

We assume that the heat kernel $\HG$ is compactly supported near the lower left corner of the front face. Its asymptotic behaviour is appropriately described in the projective coordinates
\[
\tau=\frac{t}{x^2}, \ s=\frac{\wx}{x}, \ u=\frac{y-\widetilde{y}}{x}, \ x, \ y, \ z, \ \widetilde{z},
\]
where $\tau, s$, and $x$ are the defining functions of tf, lf and ff, respectively. The coordinates are valid whenever $(\tau, s, u)$ are bounded as $(t,x,\wx)$ approach zero. The edge vector fields obey the following transformation rules:
\[
\beta^*(x\partial_x)=-2 \tau \partial_{\tau} - s \partial_s - u \partial_u + x\partial_x, \ \beta^*(x\partial_y)=\partial_u + x \partial_y, \ \beta^*(\partial_z)=\partial_z.
\]
Note that we abuse notation by writing $x \partial_x, x \partial_y$ vector fields on the right side of these expressions.  By Theorems \ref{thm-Fried} and \ref{thm-incomplete} we infer that for any $n\in \N_0$, 
\begin{align*}
\beta^*(V_e^{k+n} \HF)(\tau, x, y, z, s, u, \wz)&=x^{-1-b} s^{\nu_{\min}+1/2} G(\tau, x, y, z, s, u, \wz), \\
\beta^*(V_e^{k+n} E^{ij}_\Gamma)(\tau, x, y, z, s, u, \wz)&=x^{-1-b+\nu_i + \nu_j} s^{-\nu_i+1/2} G'(\tau, x, y, z, s, u, \wz),
\end{align*}
where $G,G'$ are bounded in their entries and infinitely vanishing as $\tau$ goes to zero, with the same compact support as $\beta^*\HG$. For the transformation rule of the volume form we compute

\[
\beta^*(d\wx \dvb(\wx))=h \cdot x^{1+b} ds\, du\, d\wz,
\]
where $h$ is a bounded distribution on $\mathscr{M}_h^2$.
Hence we can write for any $\w \in C^k_e(\mathscr{U}, \Lambda^p T^*\mathscr{U})^\Phi$
\begin{equation}
\label{lf-est}
\begin{split}
\int V_e^{k+n} (x/\wx)^{-\beta} \HF \w \, d\wx \, d\wy \, d\wz 
&=\int s^{\nu_{\min}+1/2 +\beta} G \w ds \, du \, d\wz, \\
\int V_e^{k+n} (x/\wx)^{-\beta} E^{ij}_\Gamma \w \, d\wx \, d\wy \, d\wz 
&=\int s^{1/2 +\beta-\nu_i} x^{\nu_i+\nu_j} G' \w  ds \, du \, d\wz.
\end{split}
\end{equation}
By the choice of the range for the weight $\beta$, the expressions above are integrable and we find
\[
\| \HG (x^\beta \w) \|_{x^{\beta}C^{k+n}_e} \leq C\, \| \w \| \leq 
C\, \| x^\beta \w \|_{x^{\beta}C^{k}_e}.
\]
Note that this estimate also accounts for the second statement of the theorem, since 
for $\w \in C^\A_e(\mathscr{U}, \Lambda^p T^*\mathscr{U})^\Phi$ we have 
$\|\w\| \leq \|\w \|_\A$. 

\bigskip

\underline{\emph{Estimates near the lower right corner of the front face.}} \ \bigskip

Applying our argument for the left corner of the front face near the right corner appears to produce singular behaviour in powers of the defining function for the front face.  However, the argument is not symmetric in the $(x,y,z)$ versus $(\wx, \wy, \wz)$ variables as these variables play different roles in the integration \eqref{VH-incomp}. Thus we may use a second change of variables to take advantage of the infinite order decay at the temporal face. To begin, we use the projective coordinates
\[
\wtau=\frac{t}{\wx^2}, \ \ws=\frac{x}{\wx}, \ \wu=\frac{y-\widetilde{y}}{\wx}, \ z, \ \wx, \ \wy, \ \widetilde{z},
\]
where $\wtau, \ws,$ and $\wx$ are the defining functions of tf, rf and ff, respectively. The coordinates are valid whenever $(\wtau, \ws, \wu)$ are bounded as $(t,x,\wx)$ approach zero. The edge vector fields obey the following transformation rules:
\[
\beta^*(x\partial_x)=  \ws \partial_{\ws}, \ \beta^*(x\partial_y)= \ws \partial_{\wu}, \ \beta^*(\partial_z)=\partial_z.
\]
By Theorems \ref{thm-Fried} and \ref{thm-incomplete} we infer that for any $n\in \N_0$, 
\begin{align*}
\beta^*(V_e^{k+n} \HF)(\widetilde{\tau}, \widetilde{s}, \widetilde{u}, z, \wx, \wy, \wz)
&=\wx^{-1-b} \widetilde{s}^{\nu_{\min}+1/2} G(\widetilde{\tau}, \widetilde{s}, \widetilde{u}, z, \wx, \wy, \wz), \\
\beta^*(V_e^{k+n} E^{ij}_\Gamma)(\widetilde{\tau}, \widetilde{s}, \widetilde{u}, z, \wx, \wy, \wz)
&=\wx^{-1-b+\nu_i + \nu_j} \widetilde{s}^{-\nu_j+1/2} G'(\widetilde{\tau}, \widetilde{s}, \widetilde{u}, z, \wx, \wy, \wz),
\end{align*}
where $G,G'$ are bounded in their entries and infinitely vanishing as $\widetilde{\tau}$ goes to zero, with the same compact support as $\beta^*\HG$. For the transformation rule of the volume form we compute

\[
\beta^*(d\wx \dvb(\wx))=h \cdot \wx^{1+b} \widetilde{\tau}^{-1} d\widetilde{\tau}\, d\widetilde{u}\, d\wz,
\]
where $h$ is a bounded distribution on $\mathscr{M}_h^2$.
Hence we can write for any $\w \in C^k_e(\mathscr{U}, \Lambda^p T^*\mathscr{U})^\Phi$
\begin{align*}
\int V_e^{k+n} (x/\wx)^{-\beta} \HF \w \, d\wx \, d\wy \, d\wz 
&=\int \wtau^{-1} \ws^{-\beta+\nu_{\min}+1/2} G \w d\widetilde{\tau}\, d\widetilde{u}\, d\wz, \\
\int V_e^{k+n} (x/\wx)^{-\beta} E^{ij}_\Gamma \w \, d\wx \, d\wy \, d\wz 
&=\int \wtau^{-1} \ws^{ -\beta-\nu_j+1/2} \wx^{\nu_i+\nu_j}  G' \w d\widetilde{\tau}\, d\widetilde{u}\, d\wz.
\end{align*}
By the choice of the range for the weight $\beta$, the exponent on $\ws$ is at least $0$.  Since $G$ and $G'$ are infinitely vanishing as $\widetilde{\tau}$ goes to zero, we find
\[
\| \HG (x^\beta \w) \|_{x^{\beta}C^{k+n}_e} \leq C\, \| \w \| \leq C\, \| x^\beta \w \|_{x^{\beta}C^{k}_e}.
\]
Note that this estimate also accounts for the second statement of the theorem, since 
for $\w \in C^\A_e(\mathscr{U}, \Lambda^p T^*\mathscr{U})^\Phi$ we have 
$\|\w\| \leq \|\w \|_\A$. \bigskip

\underline{\emph{Estimates near the top corner of the front face}} \ \bigskip

The asymptotic behaviour of the heat kernel $\HG$, compactly supported near the top corner of ff, is appropriately described in the projective coordinates
\[
\rho=\sqrt{t}, \  \xi=\frac{x}{\rho}, \ \widetilde{\xi}=\frac{\widetilde{x}}{\rho}, \ u=\frac{y-\widetilde{y}}{\rho}, \ y, \ z, \ \widetilde{z},
\]
where $\rho, \xi$, and $\widetilde{\xi}$ are the defining functions of the faces ff, rf and lf, respectively. The coordinates are valid whenever $(\rho, \xi,\widetilde{\xi})$ are bounded as $(t,x,\wx)$ approach zero. The edge vector fields obey the following transformation rules:
\[
\beta^*(x\partial_x)=\xi\partial_{\xi}, \ \beta^*(x\partial_y)=\xi \partial_u + \rho \xi \partial_y, \ \beta^*(\partial_z)=\partial_z.
\]
By Theorems \ref{thm-Fried} and \ref{thm-incomplete} we infer that for any $n\in \N_0$, 
\begin{align*}
\beta^*(V_e^{k+n} \HF)( \rho, \xi, y, z, \widetilde{\xi}, u, \widetilde{z})&=
\rho^{-1-b} (\xi \widetilde{\xi})^{\nu_{\min}+1/2} G( \rho, \xi, y, z, \widetilde{\xi}, u, \widetilde{z}), \\
\beta^*(V_e^{k+n} E^{ij}_\Gamma)( \rho, \xi, y, z, \widetilde{\xi}, u, \widetilde{z})
&=\rho^{-1-b+\nu_i + \nu_j} \xi^{-\nu_j+1/2} \widetilde{\xi}^{-\nu_i + 1/2} 
G'( \rho, \xi, y, z, \widetilde{\xi}, u, \widetilde{z}),
\end{align*}
where $G,G'$ are bounded in their entries, with the same compact support as $\beta^*\HG$. 
For the transformation rule of the volume form we compute

\[
\beta^*(d\wx \dvb(\wx))=h \cdot \rho^{1+b} d\widetilde{\xi}\, du\, d\wz,
\]
where $h$ is a bounded distribution on $\mathscr{M}_h^2$.
Hence we can write for any $\w \in C^k_e(\mathscr{U}, \Lambda^p T^*\mathscr{U})^\Phi$
\begin{align*}
\int V_e^{k+n} (x/\wx)^{-\beta} \HF \w \, d\wx \, d\wy \, d\wz 
&=\int \xi^{\nu_{\min}+1/2 -\beta} \widetilde{\xi}^{\nu_{\min}+1/2+\beta}G \w  d\widetilde{\xi}\, du\, d\wz, \\
\int V_e^{k+n} (x/\wx)^{-\beta} E^{ij}_\Gamma \w \, d\wx \, d\wy \, d\wz 
&=\int \xi^{-\nu_j +1/2 - \beta} \widetilde{\xi}^{-\nu_i +1/2 + \beta}\rho^{\nu_i+\nu_j} 
G' \w d\widetilde{\xi}\, du\, d\wz.
\end{align*}
By the choice of the range for the weight $\beta$, the expressions above are integrable and we find
\[
\| \HG (x^\beta \w) \|_{x^{\beta}C^{k+n}_e} \leq C\, \| \w \| \leq C\, \| x^\beta \w \|_{x^{\beta}C^{k}_e}.
\]
Note that this estimate also accounts for the second statement of the theorem, since 
for $\w \in C^\A_e(\mathscr{U}, \Lambda^p T^*\mathscr{U})^\Phi$ we have 
$\|\w\| \leq \|\w \|_\A$. \bigskip

\underline{\emph{Estimates where the diagonal meets the front face}} \ \bigskip

The asymptotic behaviour of the heat kernel with compact support in the neighbourhood where td meets ff is appropriately described in the projective coordinates
\begin{align}\label{td2}
\eta^2=\frac{t}{x^2}, \ S=\frac{x-\wx}{\sqrt{t}}, \ U=\frac{y-\wy}{\sqrt{t}}, \ Z=\frac{x(z-\wz)}{\sqrt{t}}, \ x, \ y, \ z
\end{align}
where in these coordinates $\eta$ and $x$ are the defining functions of td and ff, respectively. The edge vector fields obey the following transformation rules:
\[
\beta^*(x\partial_x)=-\eta \partial_{\eta} +\frac{1}{\eta}\partial_S + Z \partial_Z + x\partial_x, \ \beta^*(x\partial_y)=\frac{1}{\eta}\partial_U + x \partial_y, \ \beta^*(\partial_z)=\frac{1}{\eta}\partial_Z + \partial_z.
\]
By Theorems \ref{thm-Fried} and \ref{thm-incomplete} we infer that for any $n\in \N_0$, 
\begin{align*}
\beta^*(V_e^{n} \HF)(\eta, S, U, Z, x, y, z)&=x^{-1-b} \eta^{-m-n} G(\eta, S, U, Z, x, y, z), \\
\beta^*(V_e^{n} E^{ij}_\Gamma)(\eta, S, U, Z, x, y, z)&=x^{-1-b+\nu_i + \nu_j} G'(\eta, S, U, Z, x, y, z),
\end{align*}
where $G,G'$ are bounded in their entries and $G'$ is infinitely vanishing as $\eta$ goes to zero, 
with the same compact support as $\beta^*\HG$. For the transformation rule of the volume form we compute

\[
\beta^*(d\wx \dvb(\wx))=h \cdot x^{1+b} \eta^m dS\, dU\, dZ,
\]
where $h$ is a bounded distribution on $\mathscr{M}_h^2$.
Note that the factor $(x/\wx)^{-\beta}$ arising from the weighting of the spaces is bounded and can be omitted near the diagonal. Hence we can write for any $\w \in C^k_e(\mathscr{U}, \Lambda^p T^*\mathscr{U})^\Phi$
\begin{align*}
\int V_e^{n} \HF \w \, d\wx \, d\wy \, d\wz 
&=\int \eta^{-n}  G \w dS\, dU\, dZ, \\
\int V_e^{n} E^{ij}_\Gamma \w \, d\wx \, d\wy \, d\wz 
&=\int x^{\nu_i+\nu_j} \eta^m G' \w  dS\, dU\, dZ.
\end{align*}
The expressions above are integrable and we find
\[
\| \HG (x^\beta \w) \|_{x^{\beta}C^{n}_e} \leq C\, t^{-n/2}\, \| \w \| = C\,  t^{-n/2}\, \| x^\beta \w \|_{x^{\beta}C^{0}_e},
\]
This completes the proof of estimate \eqref{est:Ck-spaces-inc} for $k=0$. 

For the general case, it suffices to consider the impact of $k+n$ edge derivatives
applied to $\HF$. The estimates for the $E^{ij}_\Gamma$ components do not change, 
since $\beta^*(E^{ij}_\Gamma)$ vanishes to infinite order at the temporal diagonal $\td$. 
Transforming the edge derivatives into projective coordinates near where the diagonal meets the front face, we see that the only possible singular behaviour comes from components of the form $\eta^{-1}\partial_{S}, \eta^{-1}\partial_{U}$ and $\eta^{-1}\partial_{Z}$. We consider $\eta^{-1}\partial_{S}$ (the others are similar) and estimate the integral
\begin{align*}
F:&=\int \eta^{-1} \partial_{S} \beta^* \HF(\eta, S,U,Z,x,y,z) \w  x^{1+b} \eta^m dS\, dU\, dZ\\
&= \int \eta^{-1} \partial_{S}G \w\left(x(1-\eta S), y-x\eta U, z-\eta Z \right)  dS\, dU\, dZ.
\end{align*}
We perform integration by parts, where the boundary terms lie away from the diagonal and hence are infinitely vanishing for $t\to 0$ by the asymptotic behaviour of the heat kernel. Omitting these irrelevant terms we obtain
\begin{align*}
F&=\int  G \eta^{-1}\partial_{S} \w\left(x(1-\eta S), y-x\eta U, z-\eta Z \right)  dS\, dU\, dZ\\
&= \int  -\frac{1}{(1 - \eta S)} G \left( \wx \partial_{\wx}\, \w\right) \left(x(1-\eta S), y-x\eta U, z-\eta Z \right)  dS\, dU\, dZ.
\end{align*}
Note that $\frac{1}{1 - \eta S} = \frac{x}{\wx}$, which is bounded near td.
Since $G (\eta, S,U,Z,x,y,z)$ is bounded in its entries, and in fact infinitely vanishing as $|(S, U, Z)| \to \infty$, we can estimate
\[
\| \HF (x^\beta \w) \|_{x^{\beta}C^{1}_e} \leq C\, \| x^\beta \w \|_{x^{\beta}C^{1}_e}.
\]

Repeating the same argument for $k>1$ edge derivatives we find that the heat operator is bounded in its action on 
$x^\beta C^k_e(\mathscr{U}, \Lambda^p T^*\mathscr{U})^\Phi$.  
In view of the proof for $k=0$, this proves the estimate \eqref{est:Ck-spaces-inc} in full generality and it remains to check \eqref{est:holder-spaces-inc}. 
Again, due to the infinite order vanishing of each $\beta^*E^{ij}_\Gamma$ at the temporal diagonal $\td$, 
it suffices to prove the statement for $\HF$.
Consider
\begin{align*}
&\int \eta^{-2} \partial^2_{S}\HF \w \ d\wx \, d\wy \, d\wz \\ = \, &
 \int \eta^{-2}\partial^2_{S}G \left(\w\left(x(1-\eta S), y-x\eta U, z-\eta Z \right) -\w(x,y,z)\right) dS\, dU\, dZ \\ +\, &
 \int \eta^{-2}\partial^2_{S}G \w\left(x,y,z\right) dS\, dU\, dZ :=E_1 + E_2.
\end{align*}
Using the H\"older condition on $\w$ we deduce for the first integral (read the estimates componentwise)
\begin{align*}
|E_1| &\leq \|\w\|_\A \, \int \eta^{-2}  
\left( \left|x\eta S\right|^{\A} + \left|x\eta U\right|^{\A} + \left|(2-\eta S)x\eta Z\right|^{\A} \right)
|G''| dS\, dU\, dZ \\ &\leq  \|\w\|_\A \int \eta^{-2+\A}  
\left( \left|S\right|^{\A} + \left|U\right|^{\A} + \left|Z\right|^{\A} \right)
|G''| dS\, dU\, dZ \leq C\, t^{-1+\A/2} \|\w\|_\A. 
\end{align*}
where $G''$ is some bounded polyhomogeneous distribution, vanishing to infinite order as $|(S,U,Z)|\to \infty$.
For the second integral $E_2$ we use the fact that  $G$ vanishes to infinite order as $|(S,U,Z)|\to \infty$ and employ Stokes' 
theorem for $dS$ integration to deduce the statement.

We have proved Theorem \ref{thm:main:incomplete} for the heat operator acting on $x^{\beta}C^*_e(M,{}^{ie}\Lambda^p M)$. 
In order to complete the proof for the heat kernel acting on $x^\beta \mathscr{C}^k_e$ spaces, it suffices to prove that the image of 
$x^\beta C^k_e$ under the heat operator lies in $x^\beta \mathscr{C}_e^{k+2}$. For any fixed $t>0$ we may write in standard coordinates $((x,y,z),(\wx,\wy,\wz))$
\begin{align*}
\int V_e^{k+n} (x/\wx)^{-\beta} \HF \w \, d\wx \, d\wy \, d\wz 
&=\int x^{\nu_{\min}+1/2 -\beta} \widetilde{x}^{\nu_{\min}+1/2+\beta}G \w  d\widetilde{x}\, d\wy\, d\wz, \\
\int V_e^{k+n} (x/\wx)^{-\beta} E^{ij}_\Gamma \w \, d\wx \, d\wy \, d\wz 
&=\int x^{-\nu_i +1/2 - \beta} \widetilde{x}^{-\nu_j +1/2 + \beta}
G' \w d\widetilde{x}\, d\wy\, d\wz.
\end{align*}
Thus for $\beta$ away from the right boundary of the weight interval, 
the integrals above are $O(x^\epsilon)$ as $x\to 0$ for some $\epsilon >0$.
Consequently, $e^{-t\Delta}\left(x^\beta C^k_e\right)\subset x^\beta \mathscr{C}_e^{k+2}$.

It remains to prove Theorem \ref{thm:main:incomplete-functions}.
This is a direct consequence of Proposition \ref{heat-expansion}, which asserts $z$-independence of the leading coefficient in the right face expansion of $H$. 
Note that for any polyhomogeneous $\w \sim \sum x^\gamma a_\gamma(y,z)$ as $x\to 0$ with $\gamma\geq 0$ in a discrete set and $a_0$ independent of $z$,
the edge derivatives $x\partial_x \w, x\partial_y \w$ and $\partial_z \w$ vanish identically at $x=0$.
Consequently, for any fixed $t>0$ and any $f\in C^k_e(M)$, Proposition \ref{heat-expansion} implies $((V_1 \dots V_{k+2}) e^{-t\Delta}f)(0,y,z)\equiv 0$ for any 
$V_j\in \mathcal{V}_e, j=1,\dots,k+2$. Hence $e^{-t\Delta}\left(C^k_e(M)\right)\subset \mathscr{C}_e^{k+2}(M)$.
This proves Theorem \ref{thm:main:incomplete-functions}.

\subsection{Estimates on spaces with a complete edge metric}  \label{est:complete}

In this section we prove Theorem \ref{thm:main:complete}.  
Once again the main task in the proof is to check the estimates, as the continuity of $e^{-t \Delta} \w$ and its derivatives are easy to obtain by the dominated convergence theorem in each of the coordinate systems that will be discussed.  The proof of the mapping properties amounts to estimating edge derivatives of the heat kernel lifted to $\mathscr{M}^2_h$; these estimates are made locally in various coordinate patches of $\mathscr{M}^2_h$.

The estimates for the heat kernel supported near the diagonal away from the front face or supported in the interior of $\mathscr{M}^2_h$ are known (e.g., \cite{LadSolUra:LQE}), so our discussion reduces to the case of the heat kernel compactly supported in an open neighbourhood of the front face.  Within this neighbourhood, we consider the heat kernel being supported in each of three coordinate charts: one near each of the left and right corners of the front face, for which we use coordinates \eqref{left-coord}, and one near the intersection of td and ff, for which we use coordinates \eqref{diag-coord}.  Although the argument is not completely symmetric in the left and right corners due to the different roles played in the integration by tilded and untilded variables, the estimation is analogous, as we indicate below.  These three charts correspond to different asymptotic regimes as the variables $(t, x, \tilde{x})$ approach zero. 

Before we specialize to the coordinate charts, we note that the estimates of Theorem \ref{thm:main:complete} for weighted spaces follow from the estimates for unweighted spaces.  Indeed, proving the estimates for weighted spaces is equivalent to proving the unweighted estimates for the conjugated operator $X^{-\beta} e^{-t \Delta} X^{\beta}$, where $X$ refers to the obvious multiplication operator.  Written as an integral operator, we have
\begin{equation*}
e^{-t\Delta}\w(x,y,z) = \int_M x^{-\beta} H\left( t, (x,y,z), (\wx, \wy, \wz) \right) \wx^{\beta} \cdot \w(\wx,\wy,\wz) \wx^{-b-1} d\wx \dvb(\wx). 
\end{equation*}
One can check that in the coordinate systems \eqref{left-coord} and \eqref{diag-coord}, the conjugated integral kernel $x^{-\beta} H \wx^{\beta}$ has the same asymptotics as $H$ at all boundary hypersurfaces of $\mathscr{M}^2_h$.  Thus we may assume $w=0$ without loss of generality. \bigskip

\underline{\emph{Estimates near the lower left corner of the front face}} \ \bigskip

We begin by assuming that the heat kernel is supported near the left-hand corner and work in the coordinates \eqref{left-coord}
\[
\tau = \sqrt{t}, \, \ s = \frac{\wx}{x}, \, \ u = \frac{y - \myy}{x}, \, x,\, y,\, z, \, \wz.
\]
In these coordinates $\tau, s,$ and $x$ are the defining functions of tf, lf and ff respectively. The coordinates are valid whenever $(\tau, s, x)$ are bounded as $(t,x,\wx)$ approach zero. The edge vector fields obey the following transformation rules:
$$\beta^*(x\partial_x)=-s\partial_s - u \partial_u + x \partial_x, \ \beta^*(x\partial_y)=\partial_u + x \partial_y, \ \beta^*(\partial_z)=\partial_z.$$
Hence by Theorem \ref{thm:Albin-htkrnl} we find that for any $n\in \N_0$, $\beta^*(V_e^nH)(\tau, s, u, x,y,z,\wz)$
is bounded in $x$ and vanishing to infinite order as $\tau \to 0$ and $s \to 0$.  For the transformation rule of the volume form we compute 
$$\beta^*(\wx^{-b-1}d\wx \dvb(\wx))=h \cdot s^{-b-1}ds\, du\, d\wz,$$
where $h$ is a bounded distribution on $\mathcal{M}_h^2$. Hence for any $\w\in C^k_e(M,{}^{e}\Lambda^p M)$, we have
\begin{align*}
 V_e^ne^{-t\Delta}\w(x,y,z)=\int h\cdot \beta^*(V_e^nH)(\tau, s, u, x,y,z,\wz) \w(sx,y-xu,\wz)s^{-b-1} ds\, du\, d\wz.
\end{align*}
Since $\beta^*(V_e^nH)(\tau, s, u, x,y,z,\wz)$ is vanishing to infinite order as $s \to 0$, we find
$$ \|e^{-t\Delta}\w\|_{C^{k+n}_e}\leq C \|\w\| \leq C \cdot \|\w\|_{C_e^k},$$
for any $k\in \N_0$. The second statement of Theorem \ref{thm:main:complete} follows as well, since for 
$\w \in C_e^{\A}(M,{}^{e}\Lambda^p M)$ we have $\|\w\|\leq \|\w\|_\A$. 

\bigskip

\underline{\emph{Estimates near the lower right corner of the front face}} \ \bigskip

In the complete case the estimates near the right corner are simplified by the infinite order vanishing of the heat kernel at the right face.  We begin by assuming that the heat kernel is supported near the right-hand corner and work in the coordinates \eqref{left-coord}
\[
\tau = \sqrt{t}, \, \ \ws = \frac{x}{\wx}, \, \ \wu = \frac{y - \myy}{\wx}, \, \wx,\, \wy,\, z, \, \wz.
\]
In these coordinates $\tau, \ws,$ and $\wx$ are the defining functions of tf, rf and ff respectively. The coordinates are valid whenever $(\tau, \ws, \wx)$ are bounded as $(t,x,\wx)$ approach zero. The edge vector fields obey the following transformation rules:
$$\beta^*(x\partial_x)= \ws\partial_{\ws}, \ \beta^*(x\partial_y)=\ws \partial_{\wu}, \ \beta^*(\partial_z)=\partial_z.$$
Hence by Theorem \ref{thm:Albin-htkrnl} we find that for any $n\in \N_0$, $\beta^*(V_e^nH)(\tau, \ws, \wu, \wx,\wy,z,\wz)$ is bounded in $\wx$ and vanishing to infinite order as $\tau \to 0$ and $\ws \to 0$.  For the transformation rule of the volume form we compute 
$$\beta^*(\wx^{-b-1}d\wx \dvb(\wx))=h \cdot \ws^{-1} d\ws\, d\wu\, d\wz,$$
where $h$ is a bounded distribution on $\mathcal{M}_h^2$. Hence for any $\w\in C^k_e(M,{}^{e}\Lambda^p M)$, we have
\begin{align*}
 V_e^ne^{-t\Delta}\w(x,y,z)=\int h\cdot \beta^*(V_e^nH)(\tau, \ws, \wu, \wx,\wy,z,\wz) \w(x/\ws,y-\wx \wu,\wz) \ws^{-1} d\ws\, d\wu\, d\wz.
\end{align*}
Since $\beta^*(V_e^nH)(\tau, \ws, \wu, \wx,\wy,z,\wz)$ is vanishing to infinite order as $\ws \to 0$, we find
$$ \|e^{-t\Delta}\w\|_{C^{k+n}_e}\leq C \|\w\| \leq C \cdot \|\w\|_{C_e^k},$$
for any $k\in \N_0$. The second statement of Theorem \ref{thm:main:complete} follows as well, since for 
$\w \in C_e^{\A}(M,{}^{e}\Lambda^p M)$ we have $\|\w\|\leq \|\w\|_\A$. 

\bigskip

\underline{\emph{Estimates near the diagonal}} \ \bigskip

The asymptotic behaviour of the heat kernel with compact support in the neighbourhood where the diagonal meets the front face is appropriately described in the projective coordinates \eqref{diag-coord}
\[
\eta=\sqrt{t}, \ S=\frac{x-\wx}{x\sqrt{t}}, \ U=\frac{y-\wy}{x\sqrt{t}}, \ Z=\frac{z-\wz}{\sqrt{t}}, \ x, \ y, \ z .
\]
Recall that in these coordinates $\eta$ and $x$ are the defining functions of td and ff, respectively. The edge vector fields obey the following transformation rules:
\[
\beta^*(x\partial_x)=(1-\eta S)\frac{1}{\eta}\partial_S -U\partial_U + x\partial_x, \ \beta^*(x\partial_y)=\frac{1}{\eta}\partial_U + x \partial_y, \ \beta^*(\partial_z)=\frac{1}{\eta}\partial_Z + \partial_z.
\]
Hence by Theorem \ref{thm:Albin-htkrnl} we infer that for any $n\in \N_0$,
\[
\beta^*(V_e^n H)(\eta, S, U, Z, x, y, z)=\eta^{-m-n} G(\eta, S, U, Z, x, y, z), 
\]
where $G$ is bounded in its entries, and in fact infinitely vanishing as $|(S, U, Z)|\to \infty$, with the same compact support as $\beta^*H$. The coordinates $(\wx,\wy,\wz)$ on the second copy of $M$ are expressed in terms of the projective coordinates \eqref{diag-coord} by
$$\wx=x(1-\eta S), \ \wy=y-x\eta U, \ \wz=z-\eta Z.$$
Thus the transformation rule of the volume form is
\[
\beta^*(\wx^{-b-1} d\wx \dvb(\wx))=h \eta^m (1-\eta S)^{-b-1} dS\, dU\, dZ,
\]
where $h$ is a bounded distribution on $\mathscr{M}^2_h$. Hence for any $\w\in C_e^0(M,{}^{e}\Lambda^p M)$, we have
\begin{align*}
D:&=\int_M V_e^{n} H\left( t, (x,y,z), (\wx, \wy, \wz) \right) \w(\wx,\wy,\wz) \wx^{-b-1} d\wx \dvb(\wx) \\ &=\int \eta^{-n} (1-\eta S)^{-b-1} h\cdot G\,  \w(x(1-\eta S), y-x\eta U, z-\eta Z) \, dS\, dU\, dZ .
\end{align*}
Estimating $ (1-\eta S)^{-b-1} h\cdot G$ against a constant, we find
\[
|D| \leq C \cdot t^{-n/2} \|\w\|.
\]
This completes the proof of estimate \eqref{est:Ck-spaces} for $k=0$. For the general case, let $\w\in C_e^k(M,{}^{e}\Lambda^p M)$ and consider the impact of $(k+n)$ edge derivatives. Transforming the edge derivatives into projective coordinates near the intersection of the diagonal and the front face, we see that the only possible singular behaviour comes from components of the form $\eta^{-1}\partial_{S}, \eta^{-1}\partial_{U}$ and $\eta^{-1}\partial_{Z}$. We consider $\eta^{-1}\partial_{S}$ (the others are similar) and estimate the integral
\begin{align*}
F:&=\int \eta^{-1}\left(\partial_{S}\beta^*H\right)(\eta, S,U,Z,x,y,z)\w\cdot h \eta^m (1-\eta S)^{-b-1} dS\, dU\, dZ\\
&= \int \eta^{-1}\left(\partial_{S}G\right)(\eta, S,U,Z,x,y,z)\w\cdot h (1-\eta S)^{-b-1} dS\, dU\, dZ.
\end{align*}
We integrate by parts and note that $G$ is infinitely vanishing for $S \to \pm \infty$ by the asymptotic behaviour of the heat kernel. Omitting this boundary term we obtain
\begin{align*}
F&=\int  G (\eta, S,U,Z,x,y,z) \eta^{-1}\partial_{S} \left(\w\cdot h (1-\eta S)^{-b-1}\right)  dS\, dU\, dZ\\
&= \int  G  \left[\eta^{-1}\partial_{S}\w\left(x(1-\eta S), y-x\eta U, z-\eta Z \right)\right]  h (1-\eta S)^{-b-1} dS\, dU\, dZ\\
&+ \int  G  f \cdot \left[\eta^{-1}\partial_{S}h\left(x(1-\eta S), y-x\eta U, z-\eta Z\right)(1-\eta S)^{-b-1}\right]   dS\, dU\, dZ.
\end{align*}
After differentiating in the second integral we see that no $\eta^{-1}$ term remains, and this integral can be estimated using the same approach as for $D$ above.  For the first integral note that
\[
 \eta^{-1}\partial_{S}\w\left(x(1-\eta S), y-x\eta U, z-\eta Z\right)
= -\frac{1}{(1-\eta S)}\left( \wx \partial_{\wx}\w\right) \left(x(1-\eta S), y-x\eta U, z-\eta Z\right).
\]
The edge derivative $\wx \partial_{\wx}f$ is bounded above by $\|\w\|_{C_e^1}$; the remaining terms in the integral can be estimated as for $D$ since $G (\eta, S,U,Z,x,y,z)$ is bounded in its entries and in fact infinitely vanishing as $|(S, U, Z)|\to \infty$.  Thus $|F|\leq C \|\w\|_{C_e^1}.$

Repeating the same argument for $k>1$ edge derivatives we find that the heat operator is bounded in its action on $C_e^k(M,{}^{e}\Lambda^p M)$. 
In view of the proof of estimate \eqref{est:Ck-spaces} for $k=0$, the first statement of Theorem \ref{thm:main:complete} holds in full generality and it remains to check the second statement.  Consider

\begin{align*}
&\int \eta^{-2} \partial^2_{S}  H(t,(x,y,z), (\wx, \wy, \wz))\w(\wx, \wy, \wz)\wx^{-b-1} d\wx \dvb(\wx) \\ = \, &
 \int \eta^{-2}\partial^2_{S}G \left(\w\left(x(1-\eta S), y-x\eta U, z-\eta Z\right) -\w(x,y,z)\right) h (1-\eta S)^{-b-1} dS\, dU\, dZ \\ +\, &
 \int \eta^{-2}\partial^2_{S}G \w\left(x,y,z\right) h (1-\eta S)^{-b-1} dS\, dU\, dZ =:E_1 + E_2.
\end{align*}
Using H\"older condition on $\w$ we deduce for the first integral (read the estimates componentwise)
\begin{align*}
& |E_1| \leq \|\w\|_\A \, \int \eta^{-2}  
\left( \left|\frac{\eta S}{(2-\eta S)}\right|^{\A} + \left|\frac{\eta U}{(2-\eta S)}\right|^{\A} + \left|\eta Z\right|^{\A} \right)
|G''| dS\, dU\, dZ \\ 
&\leq  \|\w\|_\A \int \eta^{-2+\A}  
\left( \left|S\right|^{\A} + \left|U\right|^{\A} + \left|Z\right|^{\A} \right)
|G''(\eta, S,U,Z,x,y,z)| dS\, dU\, dZ \leq C\, t^{-1+\A/2} \|\w\|_\A. 
\end{align*}
where $G''$ is some bounded polyhomogeneous distribution, vanishing to infinite order as $|(S,U,Z)|\to \infty$.
For the second integral $E_2$, we use the fact that  $G$ vanishes to infinite order as $|(S,U,Z)|\to \infty$ and employ Stokes' theorem for $dS$ integration to deduce the statement.
The proof of Theorem \ref{thm:main:complete} is now complete.

\section{Short-time existence of solutions to semilinear parabolic equations}\label{s-short}

In this section we provide an application of our estimates to prove short-time existence of solutions to some semilinear parabolic equations on both complete and incomplete edge spaces.  We follow \cite{JefLoy:RSH, JefLoy:RHM} closely.  The underlying idea is based on \cite{Tay:PDE}.

Consider an equation of the form
\begin{equation} \label{semilinear:IVP}
\left\{ \begin{array}{ll} \partial_t u +\Delta u &= Q(t, x, y, z, u, V_e u), \\
                            u(0) &= u_0, \end{array} \right. \end{equation}
where $\Delta$ denotes the appropriate edge Laplacian and $Q$ denotes inhomogeneous terms that involve up to one edge derivative of $u$.  We assume that $Q$ is smooth in its arguments.  For notational convenience, we set $\Phi(u) = Q(u, V_e u)$, where we have suppressed the dependence of $Q$ on the time and local space variables.  Applying Duhamel's principle, we transform the problem to an equivalent integral equation:
\[ u(t) = e^{-t \Delta} u_0 + \int_0^t e^{-(t-s) \Delta} \Phi(u(s)) ds. \]
With appropriate assumptions and the estimates of Theorems \ref{thm:main:incomplete} or \ref{thm:main:complete}, this equation may be solved by the contraction mapping principle, as we now explain.

Recall that a semigroup $S_t$ of operators on a Banach space $X$ is strongly continuous if for all $f \in X$,
\[ || S_t f - f ||_X \longrightarrow 0 \; \mbox{as} \; t \longrightarrow 0. \]

Taylor proves:
\begin{theorem} \cite[Proposition 15.1.1]{Tay:PDE}
Suppose that $X$ and $Y$ are Banach spaces such that:
\begin{enumerate}
\item $e^{-t \Delta}: X \longrightarrow X$ is a strongly continuous semigroup, for $t \geq 0$.
\item $\Phi: X \longrightarrow Y$ is locally Lipschitz.
\item $e^{-t \Delta}: Y \longrightarrow X$, for $t > 0$.
\item For some $\gamma < 1$, 
\[ ||e^{-t \Delta}|| \leq C t^{-\gamma}, \]
for $t \in (0,1]$, where $|| \cdot ||$ denotes the operator norm.
\end{enumerate}
Then for any $u_0 \in X$, the initial value problem \eqref{semilinear:IVP} has a unique solution $u \in C([0,T], X)$, for some $T > 0$.  $T$ may be estimated from below in terms of $||u_0||_{X}$.
\end{theorem}

In view of this theorem and our previous estimates, to deduce short-time existence of solutions to \eqref{semilinear:IVP} for certain locally Lipschitz $\Phi$, it remains to prove strong continuity of the semigroup $e^{-t \Delta_{\mathscr{F}}}$ on either $X=\mathscr{C}^k_e(M,g)$ or $X=C^k_e(M,g)$, as appropriate.  See \cite{Grigoryan} for a classical proof of strong continuity on compact manifolds.
 
Our argument on incomplete edge manifolds will require stochastic completeness of the heat kernel, i.e., 
\begin{align}\label{stoch-compl}
 \int_M H( t, (x,y,z), (\wx,\wy,\wz) ) \dv = 1, \; \mbox{for all} \; (x,y,z) \in M, t > 0. 
\end{align}
For incomplete edge manifolds, this is a consequence of uniqueness of solutions to the heat equation. More precisely, with $\Delta$ being 
the Friedrichs extension of the Laplace Beltrami operator with domain $\dom(\Delta)$, the solutions to the initial value problem
\begin{align*}
\partial_t u+\Delta u=0, \ u(0)=u_0\in \dom (\Delta) \subset L^2(M, \textup{vol}(g))
\end{align*}
are unique and in fact given by $u(t)=e^{-t\Delta}u_0\in \dom (\Delta)$ for any $t>0$. 
For $\dim F \geq 1$, it can be shown that $1 \in \dom(\Delta)$.  In particular, recall that the Friedrichs' extension is associated to the Lagrangian matrix $\Gamma = \textup{diag}(\psi_1^+,\dots,\psi_q^+)$.  With this $\Gamma$, the condition in Definition \ref{defn:DG} for a function $u$ to be in the Friedrichs domain simplifies to $c_j^-[u] = 0$ for $\nu_j \in [0,1)$.  Recall that $\nu_j^2$ is an eigenvalue of the operator $A$ given by \eqref{a}; for functions, $\ell = 0$ and $A = \Delta_F + (\frac{f-1}{2})^2$, so the lowest eigenvalue of $A$ is $\nu_{\min}^2 = \left(\frac{f-1}{2}\right)^2$.  Thus the leading order term in the expansion of a Friedrichs solution $u$ as in \eqref{expansion-w} is $x^{\nu_{\min} + 1/2} = x^{f/2}$.  However, we have been working under the rescaling $\Phi$, so without the rescaling any function in the Friedrichs domain has leading order term $x^0 = 1$.  The eigenfunction paired with this term in the expansion \eqref{expansion-w} is any eigenfunction of $\Delta_F$ corresponding to the eigenvalue $0$, hence a fibrewise constant function.  Clearly, then, $u\equiv 1\in \dom(\Delta)$ 
is the unique solution in $\dom (\Delta)$ to the heat equation with initial value $1$, given in terms of the 
heat operator by $u\equiv 1=e^{-t\Delta}1$. This is precisely \eqref{stoch-compl}. 

In the complete edge case, $u\equiv 1$ does not lie in the Hilbert space $L^2(M, \textup{vol}(g))$ and on general complete manifolds, solutions of the heat equation need not be unique. However by a result of Yau \cite{Yau:OHK} we have \eqref{stoch-compl} on complete manifolds with Ricci curvature bounded from below; see \cite{Hsu:HSCR} for a discussion of this and related results.  In our setting, since an edge metric can be viewed at highest order as the product of an asymptotically hyperbolic metric and a compact metric, it is straightforward to check that the Ricci curvature is bounded from below.

Using equation \eqref{stoch-compl} has the disadvantage that the proof will only work for the Laplacian on functions.  At the time of writing the authors do not know how to prove strong continuity of the Hodge Laplacian on differential $p$-forms in the incomplete case.  In the complete case an alternate proof of strong continuity for differential forms may be fashioned by a rescaling technique.  One may cover the manifold by balls of uniform $g$-radius and pull the equation back to a standard model ball in euclidean space.  This will yield uniform estimates for strong continuity near the edge.  See \cite[Appendix A]{BahuaudRicciFlow}  for similar estimation.

\begin{prop} \label{prop:incstrcty}
Let $(M,g)$ be an $m$-dimensional Riemannian manifold with a feasible incomplete edge metric. 
Let $e^{-t\Delta}$ denote the heat operator corresponding to the Friedrichs extension $\DF$ of the associated Laplacian on $(M,g)$.
Then $e^{-t \Delta}$ is strongly continuous on $\mathscr{C}^k_e(M,g)$, for any $k \geq 0$.
\end{prop}
\begin{proof}
We prove the statement by adapting the classical proof of strong continuity of the heat operator on closed (non-singular) manifolds to the present setup.
Assume first $k=0$. Using stochastic completeness of the heat kernel we find
\begin{align*}
(e^{-t \Delta } f)(p, t) - f(p) &= \int_M H\left( t, p, \widetilde{p} \right) (f(\widetilde{p})-f(p)) \dv(\widetilde{p}).
\end{align*}
Note that $f\in \mathscr{C}^0_e(M,g)$ and hence for any $\epsilon >0$ there exists some $\delta(\epsilon) >0$, such that for $d(p,\widetilde{p}) \leq \delta(\epsilon)$ one has
$|f(p)-f(\widetilde{p})| \leq \epsilon$. For any given $\epsilon >0$ we separate the integration region into 
\begin{eqnarray}\label{int_regions}
 M^+_\epsilon & := & \{\widetilde{p} \mid d(p,\widetilde{p}) \geq \delta(\epsilon)\}, \notag \\
 M^-_\epsilon & := & \{\widetilde{p} \mid d(p,\widetilde{p}) \leq \delta(\epsilon)\}.
\end{eqnarray}
Employing continuity of $f$ we find
\begin{eqnarray*}
& |e^{-t \Delta } f - f| = \left| \int_M H\left( t, p, \widetilde{p} \right) (f(\widetilde{p})-f(p)) \dv(\widetilde{p}) \right| \\
\leq &\int_{M^+} H\left( t, p, \widetilde{p} \right) |f(\widetilde{p})-f(p)| \dv(\widetilde{p}) 
+ \int_{M^-} H\left( t, p, \widetilde{p} \right) |f(\widetilde{p})-f(p)| \dv(\widetilde{p}) \\
\leq & \, 2 \frac{\sqrt{t}}{\delta(\epsilon)} \, \|f\|_{0}\int_{M^+} H\left( t, p, \widetilde{p} \right) \frac{d(p,\widetilde{p})}{\sqrt{t}} \dv(\widetilde{p})
+ \epsilon \, \int_{M^-} H\left( t, p, \widetilde{p} \right) \dv(\widetilde{p}). 
\end{eqnarray*}
The second integral is bounded by $\epsilon$. We must verify in projective coordinates that the first integral is bounded, uniformly in $(t,p, \epsilon)$.
It suffices to assume the heat kernel is supported near the front face, since away from ff the estimates reduce to the classical case of a manifold without boundary near the diagonal.
Assume first that the heat kernel is supported in a neighbourhood where td meets ff, and use the projective coordinates \eqref{td2}:
\begin{align*}
\eta^2=\frac{t}{x^2}, \ S=\frac{x-\wx}{\sqrt{t}}, \ U=\frac{y-\wy}{\sqrt{t}}, \ Z=\frac{x(z-\wz)}{\sqrt{t}}, x, y, z.
\end{align*}
In these coordinates $\eta$ and $x$ are the defining functions of td and ff, respectively. 
By Theorem \ref{thm-incomplete} we infer
\begin{equation*}
\beta^*H(\eta, S, U, Z, x, y, z)=x^{-m}\eta^{-m} G(\eta, S, U, Z, x, y, z), 
\end{equation*}
where $G$ is bounded in its entries, and in fact infinitely vanishing as $|(S, U, Z)|\to \infty$, with the same compact support as $\beta^*H$. 
The transformation rule for the volume form is
\begin{align*}
\beta^*(\wx^{f} d\wx \dvb(\wx))=h (x\eta)^m (1-\eta S)^f dS\, dU\, dZ,
\end{align*}
where $h = h\left( \eta, x(1-\eta S), y-x\eta U, z-\eta Z, x, y, z \right)$ is a bounded distribution on $\mathscr{M}^2_h$. 
We may estimate $d(p,\widetilde{p}) / \sqrt{t}$:
\begin{align*}
 \frac{d(p,\widetilde{p})}{\sqrt{t}} \leq C \sqrt{|S|^2+|U|^2+|Z|^2}.
\end{align*}
Hence we find 
\begin{align*}
\int_{M^+} H\left( t, p, \widetilde{p} \right) \frac{d(p,\widetilde{p})}{\sqrt{t}} \dv(\widetilde{p})  
\leq C' \int h (1-\eta S)^f  G \cdot \sqrt{|S|^2+|U|^2+|Z|^2}\, dS\, dU\, dZ \leq C'',
\end{align*}
where $C''$ is independent of $(t,x,y,z)$ and $\epsilon$.  Entirely analogous estimation works near the top corner of the front face.  In the estimate near the left corner (resp. right corner) of the front face one finds that $d(p,\widetilde{p})/\sqrt{t}$ is bounded by a power of $\tau^{-1}$ (resp. $\widetilde{\tau}^{-1}$), which may be absorbed into the heat kernel as it vanishes to infinite order as $\tau \rightarrow 0$ (resp. $\widetilde{\tau} \rightarrow 0$) at this corner. 
Therefore we obtain
\begin{align*}
 \| e^{-t \Delta } f - f \|_0 \leq C'' \frac{\sqrt{t}}{\delta(\epsilon)} \, \|f\|_{0} + \epsilon.
\end{align*}
Thus, for any given $\epsilon >0$ we can estimate $\| e^{-t \Delta } f - f \|_0 \leq 2\epsilon$ for $\sqrt{t} < \epsilon \delta(\epsilon) / (C'' \|f\|_{0})$.
This proves strong continuity of the heat operator on $\mathscr{C}^0_e(M,g)$. \\

We prove strong continuity on $\mathscr{C}^k_e(M,g)$ for $k\geq 1$ in a similar fashion. Near the diagonal it requires an integration by parts argument like the one in \S \ref{s:est-incomplete}. The estimates away from the diagonal are the same as for $\mathscr{C}^0_e(M,g)$, 
so we assume that $\beta^*H$ is compactly supported in a neighbourhood where td meets ff.
The edge vector fields obey the following transformation rules:
\begin{align*}
\beta^*(x\partial_x)=-\eta \partial_{\eta} +\frac{1}{\eta}\partial_S + Z \partial_Z + x\partial_x, \ \beta^*(x\partial_y)=\frac{1}{\eta}\partial_U + x \partial_y, \ \beta^*(\partial_z)=\frac{1}{\eta}\partial_Z + \partial_z.
\end{align*}
We consider $|| x \partial_x ( e^{-t \Delta} f - f )||_0$. Using stochastic completeness of the heat kernel, we find
\begin{align*}
F&:=x \partial_x ( e^{-t \Delta} f - f)
=\int (x\partial_xH) f(\wx,\wy,\wz)\wx^{f} d\wx \dvb(\wx) \\ &- \int (x\partial_x) [H f(x,y,z) \wx^{f} d\wx \dvb(\wx)]
=:F_1-F_2.
\end{align*}
Next we transform to projective coordinates and integrate by parts in $S$, where the boundary terms lie away from the diagonal and hence are infinitely vanishing for $t\to 0$ by the asymptotic behaviour of the heat kernel. Omitting these irrelevant terms,  we obtain
\begin{align*}
F_1&=\int \left(-\eta \partial_{\eta} +\frac{1}{\eta}\partial_S+Z\partial_Z +x\partial_x\right) \left[(x\eta)^{-m}G(\eta, S, U, Z, x, y, z)\right] \\ 
&\times f\left(x(1-\eta S), y-x\eta U, z-\eta Z\right) h (x\eta)^m (1-\eta S)^f dS\, dU\, dZ \\
&= \int \left[(-\eta \partial_{\eta} + Z \partial_Z + x\partial_x) (x\eta)^{-m} G\right]  \cdot f h(x\eta)^m (1-\eta S)^f dS\, dU\, dZ \\
&- \int G \left[\left(\frac{1}{\eta}\partial_S\right)f\right] h(1-\eta S)^f dS\, dU\, dZ\\
&- \int (x\eta)^{-m} G \cdot f \left[\left(\frac{1}{\eta}\partial_S \right) h(x\eta)^m(1-\eta S)^f\right] dS\, dU\, dZ.
\end{align*}
We perform similar computations for $F_2$:
\begin{align*}
F_2&=\int \left[(x\partial_x H)f(x,y,z) +H(x\partial_xf)\frac{}{}\right]\wx^f d\wx \dvb(\wx)\\
&=\int \left(\left[-\eta \partial_{\eta} +\frac{1}{\eta}\partial_S+Z\partial_Z +x\partial_x\right] (x\eta)^{-m}G\right) f\cdot h(x\eta)^m (1-\eta S)^f dS\, dU\, dZ \\ &+ \int G(\eta, S, U, Z, x, y, z)  (x\partial_xf(x,y,z)) h (1-\eta S)^f dS\, dU\, dZ \\
&=  \int \left[(-\eta \partial_{\eta} + Z \partial_Z + x\partial_x) (x\eta)^{-m}G\right]  \cdot f h(x\eta)^m (1-\eta S)^f dS\, dU\, dZ\\
&- \int (x\eta)^{-m}G \cdot f \left[\left(\frac{1}{\eta}\partial_S \right) h(x\eta)^m(1-\eta S)^f\right] dS\, dU\, dZ\\
&+\int G(\eta, S, U, Z, x, y, z)  (x\partial_xf(x,y,z)) h (1-\eta S)^f dS\, dU\, dZ.
\end{align*}
Thus $F=F_1-F_2$ becomes
\begin{align*}
F&= \int \left[(-\eta \partial_{\eta} + Z \partial_Z + x\partial_x) (x\eta)^{-m}G(\eta, S, U, Z, x, y, z)\frac{}{}\right] h(x\eta)^m (1-\eta S)^f \\ 
&\times  (f(x(1-\eta S), y-x\eta U, z-\eta Z)-f(x,y,z)) dS\, dU\, dZ \\
&- \int G(\eta, S, U, Z, x, y, z) \left[\left(\frac{1}{\eta}\partial_S\right) h \cdot (1-\eta S)^f\right] \\
&\times  (f(x(1-\eta S), y-x\eta U, z-\eta Z)-f(x,y,z))  dS\, dU\, dZ \\
&- \int G \left[\frac{1}{\eta}\partial_S f(x(1-\eta S), y-x\eta U, z-\eta Z) +x\partial_xf(x,y,z) \right] h(1-\eta S)^f dS\, dU\, dZ.
\end{align*}
Now, each of the three integrals is estimated as above for $k=0$ by separating the integration region into $M^+_\epsilon$ and $M^-_\epsilon$ for any given $\epsilon >0$.  Note that in the final integral we use the fact that $f \in \mathscr{C}^1_e(M,g)$.
Higher order and other edge derivatives may be estimated in a similar way. 
\end{proof}

\begin{prop}
Let $(M,g)$ be an $m$-dimensional Riemannian manifold with a feasible complete edge metric. 
Let $e^{-t\Delta}$ denote the heat operator of the unique self-adjoint extension $\Delta$ of the associated Laplacian on $(M,g)$.
Then $e^{-t \Delta}$ is strongly continuous on $C^k_e(M,g)$, for any $k \geq 0$.
\end{prop}
\begin{proof}
We proceed as in the incomplete edge case, again modeling the proof after the classical estimate on a closed manifold.  We estimate near td and ff; the estimation near either the left corner or right corner is simpler than what follows due to the infinite order vanishing of the heat kernel at tf.

We first consider the case $k=0$ and show
$ || e^{-t \Delta} f - f||_0 \rightarrow 0 \ \mbox{as} \ t \rightarrow 0. $
Using stochastic completeness of the heat kernel we find
\begin{align*}
(e^{-t \Delta } f)(p,t) - f(p) &= \int_{M} H\left( t, p, \widetilde{p} \right) (f(\widetilde{p})-f(p)) \dv(\widetilde{p}).
\end{align*}

Recall that $f \in C^0_e(M)$ means $f$ is continuous on a compact manifold with boundary, so for every $\epsilon > 0$ we obtain a uniform $\delta = \delta(\epsilon)$, where $d(p,\widetilde{p}) < \delta$ implies $|f(p)-f(\widetilde{p})|<\epsilon$.  We separate the integration region into $M^+_\epsilon$ and $M^-_\epsilon$ as in \eqref{int_regions} and find
\begin{align*}
|e^{-t \Delta } f & - f| = \left| \int_M H\left( t, p, \widetilde{p} \right) (f(\widetilde{p})-f(p)) \dv(\widetilde{p}) \right| \\
&\leq \int_{M^+} H\left( t, p, \widetilde{p} \right) |f(\widetilde{p})-f(p)| \dv(\widetilde{p}) + \int_{M^-} H\left( t, p, \widetilde{p} \right) |f(\widetilde{p})-f(p)| \dv(\widetilde{p}) \\
&\leq \, 2\frac{\sqrt{t}}{\delta(\epsilon)} \, \|f\|_{0}\int_{M^+} H\left( t, p, \widetilde{p} \right) \frac{d(p,\widetilde{p})}{\sqrt{t}} \dv(\widetilde{p})
+ \epsilon. 
\end{align*}

Assume that the heat kernel is supported in a neighbourhood where td meets ff, and use the appropriate projective coordinates \eqref{diag-coord}
\begin{align*}
\eta=\sqrt{t}, \ S=\frac{x-\wx}{x\sqrt{t}}, \ U=\frac{y-\wy}{x\sqrt{t}}, \ Z=\frac{z-\wz}{\sqrt{t}}, \ x, y, z .
\end{align*}
In these coordinates $\eta$ and $x$ are the defining functions of td and ff, respectively. The edge vector fields obey the transformation rules:
\begin{align*}
\beta^*(x\partial_x)=(1-\eta S)\frac{1}{\eta}\partial_S -U\partial_U + x\partial_x, \ \beta^*(x\partial_y)=\frac{1}{\eta}\partial_U + x \partial_y, \ \beta^*(\partial_z)=\frac{1}{\eta}\partial_Z + \partial_z.
\end{align*}
Hence, by Theorem \ref{thm:Albin-htkrnl} we infer
\begin{align*}
\beta^*H(\eta, S, U, Z, x, y, z)=\eta^{-m} G(\eta, S, U, Z, x, y, z), 
\end{align*}
where $G$ is bounded in its entries, and in fact infinitely vanishing as $|(S, U, Z)|\to \infty$, with the same compact support as $\beta^*H$. The coordinates $(\wx,\wy,\wz)$ on the second copy of $M$ are expressed in terms of the new projective coordinates by
$$\wx=x(1-\eta S), \ \wy=y-x\eta U, \ \wz=z-\eta Z.$$
From there we compute the transformation rule of the volume form
\begin{align*}
\beta^*(\wx^{-b-1} d\wx \dvb(\wx))=h \eta^m (1-\eta S)^{-b-1} dS\, dU\, dZ,
\end{align*}
where $h = h(\eta, x(1-\eta S), y-x\eta U, z-\eta Z, x, y, z)$ is a bounded distribution on $\mathscr{M}^2_h$ whose arguments we will suppress.  Further, we can estimate
\begin{equation*}
\frac{d(p,\widetilde{p})}{\sqrt{t}} \leq \sqrt{ \frac{S^2}{(2-\eta S)^2} + \frac{U^2}{(2-\eta S)^2} + Z^2 } 
\leq C \sqrt{ S^2 + U^2 + Z^2 }.
\end{equation*}
Hence we can write after cancellations
\begin{align*}
|e^{-t\Delta} f &- f| \leq \, 2\frac{\sqrt{t}}{\delta(\epsilon)} \, \|f\|_{0}\int_{M^+} H\left( t, p, \widetilde{p} \right) \frac{d(p,\widetilde{p})}{\sqrt{t}} \dv(\widetilde{p}) + \epsilon \\
&\leq \, C\frac{\sqrt{t}}{\delta(\epsilon)} \, \|f\|_{0} \int  h (1-\eta S)^{-b-1} \cdot G \cdot \sqrt{ S^2 + U^2 + Z^2 } dS\, dU\, dZ + \epsilon.\\
&\leq \, C \frac{\sqrt{t}}{\delta(\epsilon)} \, \|f\|_{0}  + \epsilon,
\end{align*}
where $C$ is independent of $(p,t)$ (see \S \ref{est:complete} for similar estimation near where td meets ff).  As in Proposition \ref{prop:incstrcty}, we conclude
\[ \lim_{t \rightarrow 0} || e^{-t \Delta} f - f ||_0 = 0, \]
completing the case $k=0$. 

To prove strong continuity on $C_e^k(M,g)$ for $k \geq 1$ requires an integration by parts argument similar to the one in Proposition \ref{prop:incstrcty}.  We consider $|| x \partial_x ( e^{-t \Delta} f - f )||_0$. Using stochastic completeness of the heat kernel, we find
\begin{align*}
F&:=x \partial_x ( e^{-t \Delta} f - f)
=\int (x\partial_xH) \cdot f(\wx,\wy,\wz)\wx^{-b-1} d\wx \dvb(\wx) \\ &- \int (x\partial_x) [H(t, x,y,z,\wx,\wy,\wz) f(x,y,z)] \wx^{-b-1} d\wx \dvb(\wx)
=:F_1-F_2.
\end{align*}
Next we transform to projective coordinates and integrate by parts in $S$, where the boundary terms lie away from the diagonal and hence are infinitely vanishing for $t\to 0$ by the asymptotic behaviour of the heat kernel. Omitting these irrelevant terms, we obtain
\begin{align*}
F_1&=\int \left((1-\eta S)\frac{1}{\eta}\partial_S -U\partial_U + x\partial_x\right) \left[ \eta^{-m}G(\eta, S, U, Z, x, y, z)\right] \\ 
&\times f\left(x(1-\eta S), y-x\eta U, z-\eta Z\right) h\eta^m (1-\eta S)^{-b-1} dS\, dU\, dZ \\
&= \int \left[(-U \partial_U + x\partial_x) \eta^{-m}G\right]  \cdot f h\eta^m (1-\eta S)^{-b-1} dS\, dU\, dZ \\
&- \int G \cdot (\partial_S f) \cdot h \eta^{-1} (1-\eta S)^{-b} dS\, dU\, dZ\\
&- \int G \cdot f \cdot \partial_S \eta^{-1} (h (1-\eta S)^{-b}) dS\, dU\, dZ.
\end{align*}
We perform similar computations for $F_2$:
\begin{align*}
F_2&=\int \left((x\partial_x H)f(x,y,z) +H\cdot(x\partial_xf)\frac{}{}\right)\wx^{-b-1} d\wx \dvb(\wx)\\
&=\int \left((1-\eta S)\frac{1}{\eta}\partial_S -U\partial_U + x\partial_x\right) [\eta^{-m}G] f\cdot h\eta^m (1-\eta S)^{-b-1} dS\, dU\, dZ \\ &+ \int G(\eta, S, U, Z, x, y, z)  (x\partial_xf(x,y,z)) h (1-\eta S)^{-b-1} dS\, dU\, dZ \\
&=  \int \left[ ( -U \partial_U + x\partial_x) \eta^{-m}G\right]  \cdot f(x,y,z) h\eta^m (1-\eta S)^{-b-1} dS\, dU\, dZ\\
&- \int G \cdot f \cdot \partial_S \left( \eta^{-1} h (1-\eta S)^{-b}\right) dS\, dU\, dZ\\
&+\int G(\eta, S, U, Z, x, y, z)  (x\partial_xf(x,y,z)) h (1-\eta S)^{-b-1} dS\, dU\, dZ.
\end{align*}
Carefully recalling the variables we have suppressed, we see that $F=F_1-F_2$ becomes
\begin{align*}
F&= \int \left[(-U \partial_U + x\partial_x) \eta^{-m}G(\eta, S, U, Z, x, y, z)\frac{}{}\right] h\eta^m (1-\eta S)^{-b-1} \\ 
&\times  (f(x(1-\eta S), y-x\eta U, z-\eta Z)-f(x,y,z)) dS\, dU\, dZ \\
&- \int G(\eta, S, U, Z, x, y, z) (\partial_S (\eta^{-1} h(x(1-\eta S), y-x\eta U, z-\eta Z) (1-\eta S)^{-b}) ) \\
&\times  (f(x(1-\eta S), y-x\eta U, z-\eta Z)-f(x,y,z))  dS\, dU\, dZ \\
&- \int G \cdot \left( (1-\eta S) \partial_S ( \eta^{-1} f(x(1-\eta S), y-x\eta U, z-\eta Z)) +x\partial_xf(x,y,z) \right) \cdot \\
& \times h (1-\eta S)^{-b-1} dS\, dU\, dZ.
\end{align*}
All three integrals are now estimated as in the $k=0$ case.  Note that in the second integral, the apparently singular $\eta^{-1}$ is cancelled upon first performing the differentiation; in the third integral, we use the fact that $f \in C_e^1 (M,g)$.  Higher order and other edge derivatives may be similarly estimated.
\end{proof}

\providecommand{\bysame}{\leavevmode\hbox to3em{\hrulefill}\thinspace}
\providecommand{\MR}{\relax\ifhmode\unskip\space\fi MR }

\providecommand{\MRhref}[2]{%
  \href{http://www.ams.org/mathscinet-getitem?mr=#1}{#2}
}
\providecommand{\href}[2]{#2}

\listoffigures

\end{document}